\documentclass[12pt]{article}
\usepackage{graphics,graphicx}
\usepackage{tikz}
\usepackage{amsmath}
\usepackage{indentfirst, latexsym, bm, amssymb}

\begin{document}

\newtheorem{theorem}{Theorem}[section]
\newtheorem{corollary}[theorem]{Corollary}
\newtheorem{definition}[theorem]{Definition}
\newtheorem{conjecture}[theorem]{Conjecture}
\newtheorem{question}[theorem]{Question}
\newtheorem{lemma}[theorem]{Lemma}
\newtheorem{proposition}[theorem]{Proposition}
\newtheorem{example}[theorem]{Example}
\newenvironment{proof}{\noindent {\bf
Proof.}}{\rule{3mm}{3mm}\par\medskip}
\newcommand{\remark}{\medskip\par\noindent {\bf Remark.~~}}
\newcommand{\pp}{{\it p.}}
\newcommand{\de}{\em}

\newcommand{\JEC}{{\it Europ. J. Combinatorics},  }
\newcommand{\JCTB}{{\it J. Combin. Theory Ser. B.}, }
\newcommand{\JCT}{{\it J. Combin. Theory}, }
\newcommand{\JGT}{{\it J. Graph Theory}, }
\newcommand{\ComHung}{{\it Combinatorica}, }
\newcommand{\DM}{{\it Discrete Math.}, }
\newcommand{\ARS}{{\it Ars Combin.}, }
\newcommand{\SIAMDM}{{\it SIAM J. Discrete Math.}, }
\newcommand{\SIAMADM}{{\it SIAM J. Algebraic Discrete Methods}, }
\newcommand{\SIAMC}{{\it SIAM J. Comput.}, }
\newcommand{\ConAMS}{{\it Contemp. Math. AMS}, }
\newcommand{\TransAMS}{{\it Trans. Amer. Math. Soc.}, }
\newcommand{\AnDM}{{\it Ann. Discrete Math.}, }
\newcommand{\NBS}{{\it J. Res. Nat. Bur. Standards} {\rm B}, }
\newcommand{\ConNum}{{\it Congr. Numer.}, }
\newcommand{\CJM}{{\it Canad. J. Math.}, }
\newcommand{\JLMS}{{\it J. London Math. Soc.}, }
\newcommand{\PLMS}{{\it Proc. London Math. Soc.}, }
\newcommand{\PAMS}{{\it Proc. Amer. Math. Soc.}, }
\newcommand{\JCMCC}{{\it J. Combin. Math. Combin. Comput.}, }
\newcommand{\GC}{{\it Graphs Combin.}, }

\title{ On signless Laplacian coefficients of bicyclic graphs  \thanks{
This work is supported by National Natural Science Foundation of
China (No:11271256). }}
\author{Jie Zhang, Xiao-Dong Zhang\thanks{Corresponding  author ({\it E-mail address:}
xiaodong@sjtu.edu.cn)}
\\
{\small Department of Mathematics},
{\small Shanghai Jiao Tong University} \\
{\small  800 DongChuan road, Shanghai, 200240,  P.R. China}\\
 }
\maketitle
 \begin{abstract}
  Let $G$ be a graph of order $n$ and
  $Q_G(x)= det(xI-Q(G))= \sum_{i=1}^n (-1)^i \varphi_i x^{n-i}$
  be the characteristic polynomial of the signless Laplacian matrix of a graph $G$. We
  give some transformations of $G$ which decrease all signless Laplacian coefficients
  in the set $\mathcal{B}(n)$ of all $n$-vertex bicyclic graphs. $\mathcal{B}^1(n)$ denotes
  all n-vertex bicyclic graphs with at least one odd cycle.
  We show that $B_n^1$ (obtained from $C_4$ by adding one edge between two non-adjacent vertices
  and adding $n-4$ pendent vertices at the vertex of degree $3$)
  minimizes all the signless Laplacian coefficients
  in the set $\mathcal{B}^1(n)$.
  Moreover, we prove that $B_n^2$ (obtained from $K_{2,3}$ by adding $n-5$ pendent vertices at one vertex of degree $3$)
  has minimum signless Laplacian coefficients in the set $\mathcal{B}^2(n)$
  of all $n$-vertex bicyclic graphs with two even cycles.

   \end{abstract}


\vskip 0.5cm

\section{Introduction}

Let $G$ be a simple undirect bicyclic graph. $V(G)$ and $E(G)$ denote its vertex
set and edge set, respectively. For every bicyclic graph $G$, $|E(G)|=|V(G)|+1$.
Let $d(v_i)$ denote the degree of vertex $v_i$,
and let $D(G)=diag(d(v_1),d(v_2),\cdots,d(v_n))$ be the diagonal matrix of $G$.
Furthermore, let $A(G)$ be the adjacent matrix of $G$.
The Laplacian matrix of $G$ is $L(G)=D(G)-A(G)$,
and the Laplacian characteristic polynomial is denoted by
$L_G(x)= det(xI-L(G))=\begin{matrix} \sum_{k=1}^n (-1)^k c_k x^{n-k}\end{matrix}$.
The Laplacian coefficients $c_k(G)$ of a graph $G$ can be
expressed in terms of subtree structures of $G$ by the following
result of Kelmans and Chelnokov \cite{kelmans1974}. Let $F$ be a spanning forest
of $G$ with $k$ components $T_1, T_2, \cdots, T_s$, $T_i$ has
$|V(T_i)|$ vertices, let
$$\gamma(F)=\begin{matrix} \prod_{i=1}^k |V(T_i)| \end{matrix}.$$

\begin{theorem}\label{theorem1.1}(\cite{kelmans1974})
Let $\mathcal{F}_k$ be the set of all spanning forests of $G$ with
exactly $k$ components. Then the Laplacian coefficient $c_{n-k}(G)$
is expressed by $c_{n-k}(G)=\begin{matrix} \sum_{F\in\mathcal{F}_k} \gamma(F) \end{matrix}.$
\end{theorem}

Recently, the study on the Laplacian coefficients have attracted much
attention. Mohar \cite{mohar2007} fist investigate the Laplacian coefficients of
acyclic graphs under the partial order $\preceq$. Zhang et al. \cite{zhang2009}
investigated ordering trees with diameters $3$ and $4$ by the Laplacian
coefficients. $\mbox{Ili}\acute{c}$ [13] determined the n-vertex tree of
fixed diameter which minimizes the Laplacian coefficients.
$\mbox{Ili}\acute{c}$ \cite{ilic2010} determined the n-vertex tree with given matching
number having the minimum Laplacian coefficients. He and Li \cite{he2011}
studied the ordering of all $n$-vertex trees with a perfect matching by
Laplacian coefficients. $\mbox{Ili}\acute{c}$ and
$\mbox{Ili}\acute{c}$ [12] studied the n-vertex trees with fixed pendent
vertex number and 2-degree vertex number which have minimum
Laplacian coefficients. $\mbox{Stevanovi}\acute{c}$ and
$\mbox{Ili}\acute{c}$ \cite{stevanovic2009}
investigated the Laplacian coefficients of unicyclic graphs. Tan \cite{tan2011}
characterized the determined the n-vertex unicyclic graph with given
matching number which minimizes all Laplacian coefficients. He and
Shan \cite{he2010} studied the Laplacian coefficients of bicyclic graphs.

The signless Laplacian matrix of $G$, $Q(G)=D(G)+A(G)$, which is related to $L(G)$,
has also been studied recently (see [1-5,\cite{mirzakhah2012}]).
The signless Laplacian characteristic polynomial is denoted by
$Q_G(x)= det(xI-Q(G))=\begin{matrix} \sum_{i=1}^n (-1)^i \varphi_i x^{n-i}\end{matrix}$.
Using the notation from \cite{cvetkovic2007},\cite{mirzakhah2012}, a TU-subgraph of $G$
is the spanning subgraph of $G$ whose components
are trees or odd unicyclic graphs. Assume that a TU-subgraph H of $G$ contains
$c$ odd unicyclic graphs and $s$ trees $T_1,\cdots, T_s$. The weight of H can
be expressed by $W(H)=4^c\begin{matrix} \prod_{i=1}^s n_i \end{matrix}$,
in which $n_i$ is the number of $T_i$. If $H$ contains no tree, let
$W(H)=4^c$. If $H$ is empty, in other words, $H$ does not exist, let
$W(H)=0$. The signless Laplacian coefficients $\varphi_i(G)$
can be expressed in terms of the weight of TU-subgraphs of $G$.

\begin{theorem}\label{theorem1.2}(\cite{cvetkovic2007},\cite{mirzakhah2012})
Let $G$ be a connected graph. For $\varphi_i$ as above, we have $\varphi_0=1$
and
$$\varphi_i =\sum_{H_i} W(H_i), i=1,\cdots,n;$$
where the summation runs over all TU-subgraphs $H_i$ of $G$ with $i$ edges.
\end{theorem}

From Theorem~\ref{theorem1.2},
it is obvious that for a $n$-vertex connected bicyclic graph $G$, $\varphi_1(G)=2|E(G)|=2(n+1)$.

When $G$ is non-bipartite graph, then $G$ has at least
an odd cycle $C_1$. Every TU-subgraph of $G$ with $n$
edges is obtained by deleting the edges of $E(C_2)\setminus (E(C_1)\cap E(C_2))$.
Therefore,
$$\varphi_n(G) =
\begin{cases}
|E(C_2)\setminus (E(C_1)\cap E(C_2))|,  & \mbox{if }g(C_2)\mbox{ is even} \\
\sum_{i=1}^2 |E(C_i)\setminus (E(C_1)\cap E(C_2))|, & \mbox{if }g(C_2)\mbox{ is odd}.
\end{cases}$$

When $G$ is bipartite graph, $G$ has no odd cycle, then $\varphi_n(G)=0$, and
$\varphi_{n-1}(G)$ counts the number of all spanning trees of $G$.
Every TU-subgraph of $G$ with $n-1$
edges is obtained by deleting one edge $e_1$ of $C_1$ and
one edge $e_2$ of $C_2$ $(e_1\neq e_2)$, respectively.
Thus,
$$\varphi_{n-1}(G) =
\begin{cases}
|E(C_1)||E(C_2)|-|E(C_1\cap C_2)|(|E(C_1\cap C_2)|-1),  & \mbox{if }|E(C_1\cap C_2)|\geq 1 \\
|E(C_1)||E(C_2)|, & \mbox{if }|E(C_1\cap C_2)|= 0.
\end{cases}$$
Moreover, $L(G)$ and $Q(G)$ have the same characteristic
polynomial, so $c_i(G)=\varphi_i(G), i=0,1,2,\cdots,n$, and
the expression of $\varphi_i$ in Theorem~\ref{theorem1.2} is equivalence
to the expression of $c_i$ in Theorem~\ref{theorem1.1}.

The eigenvalues of $L(G)$ and $Q(G)$ are denoted by $\mu_1(G)\geq\cdots\geq\mu_n(G)=0$
and $\nu_1(G)\geq\cdots\geq\nu_n(G)$, respectively. The incidence energy
of $G$, $IE(G)$ for short, is defined as $IE(G)=\sum_{i=1}^n \sqrt{\nu_i(G)}$
(see [7],[8],\cite{jooyandeh2009}).

Mirzakhah and Kiani \cite{mirzakhah2012} presented a connection between the incidence energy
and the signless Laplacian coefficients.

\begin{theorem}\label{theorem1.3}(\cite{mirzakhah2012})
Let $G$ and $G^\prime$ be two graphs of order $n$. If
$\varphi_i(G)\leq\varphi_i(G^\prime)$ for $1\leq i\leq n$,
then $IE(G)\leq IE(G^\prime)$ and $IE(G)< IE(G^\prime)$ if
$\varphi_i(G)< \varphi_i(G^\prime)$ for some $i$ holds.
\end{theorem}

Mirzakhah and Kiani in \cite{mirzakhah2012} gave some results about the signless
Laplacian coefficients of a graph $G$ and ordered unicyclic graphs
with fixed girth based on the signless Laplacian coefficients.
He and Shan in \cite{he2010} characterize the graph which has minimum
Laplacian coefficients among all bicyclic graphs.
Motivated by these results, we characterize the graphs which have
minimum signless Laplacian coefficients in $\mathcal{B}^1(n)$ and
$\mathcal{B}^2(n)$.

This paper is organized as follows: In the next section, we introduce
some results from the literature which are useful in this paper.
In Section 3, several transformations which simultaneously decrease all the
signless Laplacian coefficients are given.
In Section 4, we order the graphs in several sets, and in each set all graphs
have the same bases.
In Section 5, by using the results of
Section 3 and 4, we prove that $B_n^1$
has minimum signless Laplacian coefficients
in $\mathcal{B}^1(n)$, as well as incidence energy.
Meanwhile $B_n^2$
minimizes all the signless Laplacian coefficients and incidence energy
in the set $\mathcal{B}^2(n)$.

\section{Preliminaries}

Let $G$ be a graph which is not a star, let $v$ be a vertex with degree
$p+1$ in $G$, such that it is adjacent with $\{u, v_1, v_2, \cdots, v_p\}$,
where $\{v_1, v_2, \cdots, v_p\}$ are pendent vertices. The graph
$G^\prime=\sigma(G, v)$ is obtained from deleting edges $vv_1,vv_2,\cdots,vv_p$
and adding edges $uv_1,uv_2,\cdots,uv_p$.

\begin{theorem}\label{theorem2.1}(\cite{mirzakhah2012})
Let $G$ be a connected graph and $G^\prime=\sigma(G,v)$, then
$\varphi_i(G)\geq\varphi_i(G^\prime)$, for every $0\leq i\leq n$,
with equality if and only if either $i\in\{0,1,n\}$ when $G$ is
non-bipartite, or $i\in\{0,1,n-1,n\}$ otherwise.
\end{theorem}

Let $G=G_1|u : G_2|v$ be the graph obtained from two disjoint graphs
$G_1$ and $G_2$ by joining a vertex $u$ of $G_1$ and a vertex $v$ of
$G_2$ by an edge. For any graph $G$ and $v\in V(G)$, let $L_{G|v}(x)$
be the principal submatrix of $L_G(x)$ obtained by deleting the row
and column corresponding to the vertex $v$.

\begin{theorem}\label{theorem2.2}(\cite{guo2005})
If $G=G_1|u : G_2|v$, then
$L_G(x)=L_{G_1}(x)L_{G_2}(x)-L_{G_1}(x)L_{G_2|v}(x)-L_{G_2}(x)L_{G_1|u}(x)$.
\end{theorem}

\begin{theorem}\label{theorem2.3}(\cite{he2008})
If $G$ be a connected graph with $n$ vertices which consists of a subgraph $H (V(H)\geq 2)$
and $n-|V(H)|$ pendent vertices attached to a vertex $v$ in $H$, then
$L_G(x)=(x-1)^{(n-|V(H)|)}L_{H}(x)-(n-|V(H)|)x(x-1)^{(n-|V(H)|-1)}L_{H|v}(x)$.
\end{theorem}

Throughout this paper, we use the following notations. Let $\mathcal{B}(n)$
denote all bicyclic graphs with $n$ vertices.
For every graph $G\in \mathcal{B}(n)$, the lengths of the two minimal cycles $C_1, C_2$ of $G$
is denoted by $g(C_1), g(C_2)$, written by $g_1, g_2$ for short. It is obvious that
$g_1=|V(C_1)|, g_2=|V(C_2)|$. Let
$$\mathcal{B}^1(n)=\{G| G\in \mathcal{B}(n), \mbox{ at least one of } g_1, g_2
\mbox{ is odd }\}, $$
$$\mathcal{B}^2(n)=\{G| G\in \mathcal{B}(n), \mbox{ both } g_1 \mbox{ and } g_2
\mbox{ are even }\}.$$
Let $B_n^1$ denote the graph obtained from $C_4$ by adding one edge between
two non-adjacent vertices and adding $n-4$ pendent vertices at the vertex of degree $3$,
and $B_n^2$ denote the graph obtained from $K_{2,3}$ by adding $n-5$ pendent vertices
at one vertex of degree $3$, where $C_4$ is a cycle with $4$ vertices and
$K_{2,3}$ is a complete bipartite graph with $2$ and $3$ vertices in
the two sets, respectively.

Using the notations in \cite{he2008} and \cite{he2010}, we divide
$\mathcal{B}(n)$ into three types.
Let $\overline{G}$ denote the base of $G$, which is the minimal bicyclic subgraph
of $G$. It is easy to see that $\overline{G}$ can be obtained from $G$ by deleting
pendent vertices consecutively. Let $B(p,q)$ be the bicyclic graph obtained from
two vertex-disjoint cycles $C_p$ and $C_q$ by identifying vertex $u$ of $C_p$
and vertex $v$ of $C_q$. Let $B(p,l,q)$ be the bicyclic graph obtained from
two vertex-disjoint cycles $C_p$ and $C_q$ by joining vertex $u$ of $C_p$
and vertex $v$ of $C_q$ by a path $uu_1u_2\cdots u_{l-1}v$ of length $l (l\geq 1)$.
Let $B(P_k,P_l,P_m)(m\leq l\leq k)$ be the bicyclic graph obtained from three
pairwise internal disjoint paths of lengths $k,l,m$ from vertices $x$ to $y$. (see fig.1).
Define $\mathcal{B}(n)=\mathcal{B}_1(n)
\cup \mathcal{B}_2(n)\cup \mathcal{B}_3(n)$, where
$$\mathcal{B}_1(n)=\{G| G\in\mathcal{B}(n), \overline{G}=B(p,q), p\geq3, q\geq3\},$$
$$\mathcal{B}_2(n)=\{G| G\in\mathcal{B}(n), \overline{G}=B(p,l,q), p,q\geq3, l\geq1\},$$
$$\mathcal{B}_3(n)=\{G| G\in\mathcal{B}(n), \overline{G}=B(P_k,P_l,P_m), 1\leq m\leq l\leq k\}.$$
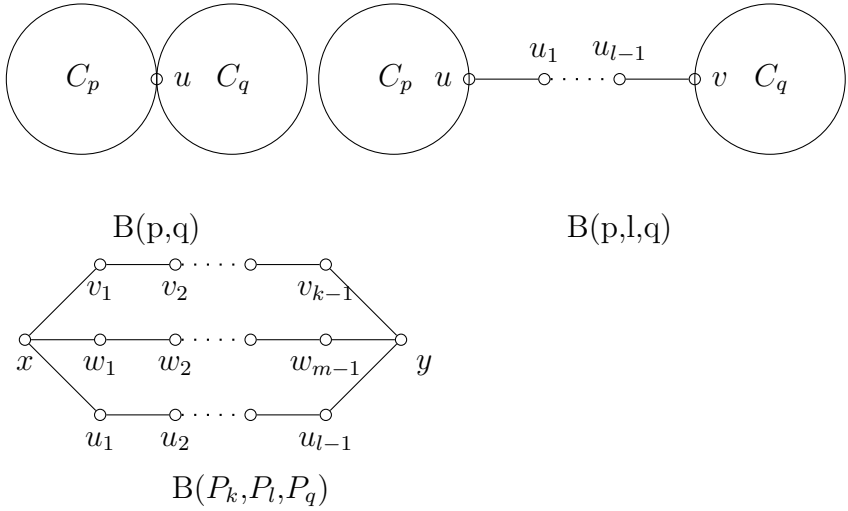
\begin{figure}

\begin{tikzpicture}
\draw  (-1,0) circle(1);
\draw  (1,0) circle(1);
\node(v900)[label=0:$C_p$]at(-1.5,0){};
\node(v901)[label=0:$C_q$]at(0.5,0){};
\node (v)[draw,shape=circle,inner sep=1.5pt,label=0:$u$] at (0,0){};
\node (t)at (0,-2){B(p,q)};
\end{tikzpicture}
\begin{tikzpicture}
\draw  (-1,0) circle(1);
\draw  (4,0) circle(1);
\node(v900)[label=0:$C_p$]at(-1.5,0){};
\node(v901)[label=0:$C_q$]at(3.5,0){};
\node (u)[draw,shape=circle,inner sep=1.5pt,label=180:$u$] at (0,0){};
\node (u1)[draw,shape=circle,inner sep=1.5pt,label=90:$u_1$] at (1,0){};
\node (u2)[draw,shape=circle,inner sep=1.5pt,label=90:$u_{l-1}$] at (2,0){};
\node (v)[draw,shape=circle,inner sep=1.5pt,label=0:$v$] at (3,0){};
\draw [loosely dotted,thick]
(u1) --(u2);
\draw (u)--(u1)(u2)--(v) ;
\node (t)at (2,-2){B(p,l,q)};
\end{tikzpicture}

\begin{tikzpicture}
\node (x)[draw,shape=circle,inner sep=1.5pt,label=-90:$x$] at (0,0){};
\node (u1)[draw,shape=circle,inner sep=1.5pt,label=-90:$u_1$] at (1,-1){};
\node (u2)[draw,shape=circle,inner sep=1.5pt,label=-90:$u_2$] at (2,-1){};
\node (u3)[draw,shape=circle,inner sep=1.5pt,label=-45:] at (3,-1){};
\node (u4)[draw,shape=circle,inner sep=1.5pt,label=-90:$u_{l-1}$] at (4,-1){};
\node (v1)[draw,shape=circle,inner sep=1.5pt,label=-90:$v_1$] at (1,1){};
\node (v2)[draw,shape=circle,inner sep=1.5pt,label=-90:$v_2$] at (2,1){};
\node (v3)[draw,shape=circle,inner sep=1.5pt,label=-45:] at (3,1){};
\node (v4)[draw,shape=circle,inner sep=1.5pt,label=-90:$v_{k-1}$] at (4,1){};
\node (w1)[draw,shape=circle,inner sep=1.5pt,label=-90:$w_1$] at (1,0){};
\node (w2)[draw,shape=circle,inner sep=1.5pt,label=-90:$w_2$] at (2,0){};
\node (w3)[draw,shape=circle,inner sep=1.5pt,label=-45:] at (3,0){};
\node (w4)[draw,shape=circle,inner sep=1.5pt,label=-90:$w_{m-1}$] at (4,0){};
\node (y)[draw,shape=circle,inner sep=1.5pt,label=-45:$y$] at (5,0){};
\draw [loosely dotted,thick]
(u2) --(u3) (v2) --(v3) (w2) --(w3);
\draw (x)--(u1)(u2)--(u1) (u3)--(u4) (u4)--(y)
(x)--(v1)(v2)--(v1) (v3)--(v4) (v4)--(y)
(x)--(w1)(w2)--(w1) (w3)--(w4) (w4)--(y);
\node (t)at (3,-2){B($P_k$,$P_l$,$P_q$)};
\end{tikzpicture}
\caption{Three types of bases of bicyclic graphs} \label{fig:pepper}
\end{figure}
\section{Transformations}
A pendent edge is an edge which is incident to a vertex of degree $1$.
Let $N_G(v)$ denote the neighbors of $v$ in the graph $G$. In this paper,
we only consider the cycles of minimal lengths in $G$ and denote the cycles
$C_1, C_2, \cdots$. Write $G_1\leq G$,
if $G_1$ is a subgraph of $G$. If two cycles $C_1, C_2$ of $G$ has the
form $B(P_k, P_l, P_m)$, then assume $|V(C_1)\cap V(C_2)|=\mbox{min}\{k,l,m\}$.

\begin{definition}\label{definition3.1}
Let $G$ be a simple connected graph with $n$ vertices, and
let $uv$ be a nonpendent edge which is not contained in the cycle
of $G$, let $G_{uv}$ obtained from $G$ by identifying vertices
$u$ and $v$ and add a new pendent edge $ww^\prime$ to the new
vertex $w$. (see fig.2).
\end{definition}

\begin{theorem}\label{theorem3.2}
Let $G$ be a n-vertex connected graph, let $G$ and $G_{uv}$ be
the two graphs presented in definition 2.1.
Then $$\varphi_i(G)\geq\varphi_i (G_{uv}),
i=0,1,\cdots,n,$$
with equality if and only if either $i\in\{0,1,n\}$ when $G$ is
non-bipartite, or $i\in\{0,1,n-1,n\}$ otherwise.
\end{theorem}

\begin{figure}
\begin{tikzpicture}
\node(v900)[label=0:$G_1$]at(-1,0){};
\node(v901)[label=0:$G_2$]at(2,0){};
\node (v0)[draw,shape=circle,inner sep=1.5pt,label=-45:$u$] at (0:0){};
\node[draw,shape=circle,inner sep=1.5pt,label=-135:$v$](v1) at (1,0){};
\draw (v0)--(v1);
\draw  (2,0) circle(1);
\draw (-1,0) circle(1);
\node(v902)[label=0:$G_1$]at(6,0){};
\node(v903)[label=0:$G_2$]at(8,0){};
\node (v0)[draw,shape=circle,inner sep=1.5pt,label=-45:$w$] at (7,0){};
\node[draw,shape=circle,inner sep=1.5pt,label=45:$w'$](v1) at (7,1){};
\draw (v0)--(v1);
\draw[->] (3.5,0)--(4.5,0);
\draw  (6,0) circle(1);
\draw (8,0) circle(1);
\end{tikzpicture}
\caption{ Transformation in Definition 3.1} \label{fig:pepper}
\end{figure}
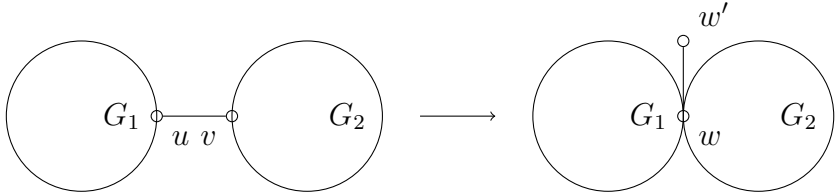

\begin{proof}
From  Theorem~\ref{theorem1.2}, according to the previous section, we have
$$\varphi_0(G)=\varphi_0(G_{uv}),
\varphi_1(G)=\varphi_1(G_{uv}).$$
Since this transformation
does not change the length of the cycles,
thus, $\varphi_n(G)=\varphi_n(G_{uv})$, and when $G$ is bipartite,
$\varphi_{n-1}(G)=\varphi_{n-1}(G_{uv})$.

When $G$ is non-bipartite,
for $2\leq i\leq n-1$, denote $\mathcal{H}_i^\prime$ and $\mathcal{H}_i$
the sets of all TU-subgraphs of $G_{uv}$ and $G$ with exactly $i$ edges,
respectively. For an arbitrary TU-subgraph $H^\prime\in\mathcal{H}_i^\prime$,
let $R^\prime$ be the component of $H^\prime$ containing $w$. Let
$N_{R^\prime}(w)\cap N_G(u)=\{u_{i_1},u_{i_2},\cdots, u_{i_r}\}$, where
$0\leq r\leq \min \{d_G(u)-1, |V(R^\prime)|-1\}$,
$N_{R^\prime}(w)\cap N_G(v)=\{v_{i_1},v_{i_2},\cdots, v_{i_s}\}$, where
$0\leq s\leq \min \{d_G(v)-1, |V(R^\prime)|-1\}$.
Define $H$ with $V(H)=V(H^\prime)-\{w,w^\prime\}+\{u,v\}$, if $ww^\prime\not\in E(H^\prime)$, $E(H)=E(H^\prime)-wu_{i_1}-\cdots-wu_{i_r}-wv_{i_1}-\cdots-wv_{i_s}+uu_{i_1}+\cdots+uu_{i_r}+vv_{i_1}+\cdots+vv_{i_s}$.
If $ww^\prime\in E(H^\prime)$, let
$E(H)=E(H^\prime)-wu_{i_1}-\cdots-wu_{i_r}-wv_{i_1}-\cdots-wv_{i_s}+uu_{i_1}+
\cdots+uu_{i_r}+vv_{i_1}+\cdots+vv_{i_s}+uv-ww^\prime$.
Let $f: \mathcal{H}_i^\prime\rightarrow \mathcal{H}_i$, and
$\mathcal{H}_i^*=f(\mathcal{H}_i^\prime)=\{f(H^\prime)|H^\prime\in \mathcal{H}_i^\prime\}$.

Now we distinguish $\mathcal{H}_i^\prime$ into the following three cases.
Denote $G_1$ the connected component containing $u$ after deleting $uv$
from $G$, and let $G_2$ be the connected component
containing $v$ after deleting $uv$ from $G$.

Case 1: $ww^\prime \in H^\prime$, then $H$ and $H^\prime$
have all the components of equal size, thus $W(H)=W(H^\prime)$.

Case 2: $ww^\prime\not\in H^\prime$, $w$ is in an odd unicyclic component
$U^\prime$ of $H^\prime$,
By the symmetry of $G_1$ and $G_2$, without loss of
generality, assume the odd cycle of $U^\prime$ is a subgraph of $G_1$.
Assume $U^\prime$ contains $a$ vertices in $G_2 \backslash \{w\}$
$(a\geq 0)$, then $W(H^\prime)=4\cdot1\cdot N$, for some constant
value $N$, $W(H)=4\cdot(a+1)\cdot N$. Thus $W(H)\geq W(H^\prime)$.

Case 3: $ww^\prime\not\in H^\prime$, $w$ is in a tree $T^\prime$ of $H^\prime$.
Assume $T^\prime$ contains $b$ vertices in $G_1\backslash\{w\}$ and
$c$ vertices in $G_2\backslash\{w\}$, then $W(H^\prime)=(b+c+1)\cdot1\cdot N$,
for some constant value $N$, $W(H)=(b+1)\cdot(c+1)\cdot N$. Thus $W(H)\geq W(H^\prime)$.

Therefore, by above discussions, $\varphi_i(G)>\varphi_i (G_{uv}),
i=2,\cdots,n-1$ holds.

When $G$ is bipartite, it is easy to prove $\varphi_i(G)>\varphi_i (G_{uv}),
i=2,\cdots,n-2$ by using above discussions of Case 1 and Case 3.
\end{proof}

\begin{remark}
When the subgraph induced by $V(G_2)$ is a star, it is easy to see
that the result of Theorem~\ref{theorem2.1} is a special case of
Theorem~\ref{theorem3.2}.

Using the transformation of Definition~\ref{definition3.1} consecutively, every
graph in $\mathcal{B}_2(n)$ can be transformed into a graph which
belongs to $\mathcal{B}_1(n)$, and keep all the signless Laplacian
coefficients not increased. Thus the graph which has minimum signless
Laplacian coefficients must belong to $\mathcal{B}_1(n)$ or $\mathcal{B}_3(n)$.
\end{remark}

\begin{definition}\label{definition3.3}
Let $G=(V,E)$ be a connected graph with at least one cycle $C_1$ $(V(C_1)\geq 5)$.
Let $u,v,w\in V(C_1)$ and $u\sim v, v\sim w$.
Assume $N_G(u)=\{v,u_1,u_2,\cdots\}$, $N_G(v)=\{u,w,v_1,v_2,\cdots\},
N_G(w)=\{v,w_1,w_2,\cdots\}$, and $N_G(u)\cap N_G(v)=\emptyset,
N_G(v)\cap N_G(w)=\emptyset, N_G(u)\cap N_G(w)=\emptyset$, then the graph
$$G^\prime=G-\{vw, ww_1,ww_2, \cdots, vv_1, vv_2, \cdots\}+\{uw, uw_1,uw_2, \cdots, uv_1, uv_2, \cdots\}.$$
\end{definition}

\begin{theorem}\label{theorem3.4}
Let $G=(V,E)$ be a connected graph with at least one cycle $C_1$ ($V(C_1)\geq 5$).
Let $u,v,w\in V(C_1)$ and $u\sim v, v\sim w$ as defined in Definition~\ref{definition3.3}.
If the following statements hold:

(1). If $\exists C_2, C_2\leq G$, such that $|V(C_1)\cap V(C_2)|\leq 2$,
then $u, v, w$ satisfy $|\{u,v,w\}\cap V(C_2)|=0,\mbox{or } 1$.

(2). If $\exists C_3, C_3\leq G$, such that $|V(C_1)\cap V(C_3)|= 3$,
then $u, v, w$ satisfy $|\{u,v,w\}\cap V(C_2)|=0,\mbox{or } 3$.

(3). If $\exists C_4, C_4\leq G$, such that $|V(C_1)\cap V(C_4)|\geq 4$,
then $u, v, w$ satisfy $|\{u,v,w\}\cap V(C_2)|=3$.

Then by performing the transformation of Definition~\ref{definition3.3}
to $u, v, w$, $\varphi_i(G)\geq\varphi_i (G^\prime)$,
$i=0,1,\cdots,n$, with equality if and only if $i\in\{0,1\}$.
\end{theorem}

\begin{proof}
$\varphi_0(G)=\varphi_0(G^\prime)$,
and since this transformation does not change the number of edges of $G$,
so $\varphi_1(G)=\varphi_1(G^\prime)$.
Next suppose $2\leq i\leq n$, denote $\mathcal{H}_i^\prime$ and $\mathcal{H}_i$
the sets of all TU-subgraphs of $G^\prime$ and $G$ with exactly $i$ edges,
respectively.

First assume $C_2$ exists, $C_3$ and $C_4$ do not exist,
and $|\{u,v,w\}\cap V(C_2)|=1$, without loss
of generality, assume $u\in V(C_1)\cap V(C_2)$ and
there is no other cycle which contains $v$ or $w$ except $C_1$.
Then all of $u,v,w$ belong to exactly one cycle $C_1$. In the discussion
below, if there is no odd cycle which satisfies some case, then we think
this case does not exist.

For an arbitrary TU-subgraph $H^\prime\in\mathcal{H}_i^\prime$,
let $R^\prime$ be the component of $H^\prime$ containing $u$. Let
$N_{R^\prime}(u)\cap N_G(w)=\{w_{i_1},w_{i_2},\cdots, w_{i_r}\}$, where
$0\leq r\leq \min \{d_G(w)-1, |V(R^\prime)|-1\}$,
$N_{R^\prime}(u)\cap N_G(v)=\{v_{i_1},v_{i_2},\cdots, v_{i_s}\}$, where
$0\leq s\leq \min \{d_G(v)-2, |V(R^\prime)|-1\}$.
Define $H$ with $V(H)=V(H^\prime)$, if $uw\not\in E(H^\prime)$,
$E(H)=E(H^\prime)-uw_{i_1}-\cdots-uw_{i_r}-uv_{i_1}-\cdots-uv_{i_s}+ww_{i_1}+\cdots+ww_{i_r}+vv_{i_1}+\cdots+vv_{i_s}$.
If $uw\in E(H^\prime)$,
$E(H)=E(H^\prime)-uw_{i_1}-\cdots-uw_{i_r}-uv_{i_1}-\cdots-uv_{i_s}+ww_{i_1}+\cdots+ww_{i_r}+vv_{i_1}+\cdots+vv_{i_s}-uw+vw$.
Let $f: \mathcal{H}_i^\prime\rightarrow \mathcal{H}_i$, and
$\mathcal{H}_i^*=f(\mathcal{H}_i^\prime)=\{f(H^\prime)|H^\prime\in \mathcal{H}_i^\prime\}$.

For convenience, write $\mathcal{H}_i^\prime$ as
$\mathcal{H}^\prime$, and $\mathcal{H}_i$ as $\mathcal{H}$.
If we include $u,v,w$ in a component of $H^\prime$,
then we have components of equal sizes in both TU-subgraphs $H^\prime$
and $H$, and thus $W(H)=W(H^\prime)$ in these cases.
Denote $\mathcal{H}^{\prime(0)}=\{H^\prime|uv\in H^\prime, vw\in H^\prime\}$.
Now we can assume that $u,v,w$ belong to $2$ or $3$ components.

We distinguish $\mathcal{H}^\prime$ into the following three cases.

Case 1: $u$ is not in an odd unicyclic
component of $H^\prime$. $H=f(H^\prime)$.
Assume $u\in T_1^\prime$, and there are $a_1+1$ vertices in
the component containing $u$ in $H-uv$ and $a_2+1$ vertices
in the component containing $w$ in $H-vw$, $a_3+1$ vertices
in the component containing $v$ in $H-uv-vw$ $(a_1, a_2, a_3\geq0)$.
Denote $N$ be the product of all the orders of components of $H^\prime$ except the components containing
$u, v, w$.

Subcase 1.1: $uv\in H^\prime, uw\not\in H^\prime$,
then $W(H^\prime)=(a_1+a_2+a_3+2)\cdot1\cdot N$, for some constant value $N$.
$W(H)=(a_1+a_3+2)\cdot(a_2+1)\cdot N$, so $W(H)-W(H^\prime)=[a_2\cdot(a_1+a_3+1)]\cdot N\geq 0$.
Denote $\mathcal{H}^{\prime(11)}=\{H^\prime|u\in T_1^\prime, uv\in H^\prime, uw\not\in H^\prime\}$.
Then $\sum_{H^\prime\in\mathcal{H}^{\prime(11)}}[W(H)-W(H^\prime)]\geq 0$.

Subcase 1.2: $uv, uw\not\in H^\prime$,
$W(H^\prime)=(a_1+a_2+a_3+1)\cdot 1\cdot1\cdot N$, for some constant value $N$.
$W(H)=(a_1+1)\cdot(a_2+1)\cdot(a_3+1)\cdot N$,
so $W(H)-W(H^\prime)\geq 0$.
Denote $\mathcal{H}^{\prime(12)}=\{H^\prime|u\in T_1^\prime, uv, uw\not\in H^\prime\}$.
Then $\sum_{H^\prime\in\mathcal{H}^{\prime(12)}}[W(H)-W(H^\prime)]\geq 0$.

Subcase 1.3: $uv\not\in H^\prime, uw\in H^\prime$,
then $W(H^\prime)=(a_1+a_2+a_3+2)\cdot1\cdot N$, for some constant value $N$.
$W(H)=(a_2+a_3+2)\cdot(a_1+1)\cdot N$, so $W(H)-W(H^\prime)=[a_1\cdot(a_2+a_3+1)]\cdot N\geq 0$.
Denote $\mathcal{H}^{\prime(13)}=\{H^\prime|u\in T_1^\prime, uv\not\in H^\prime, uw\in H^\prime\}$.
Then $\sum_{H^\prime\in\mathcal{H}^{\prime(13)}}[W(H)-W(H^\prime)]\geq 0$.

Case 2: $u$ is in an odd unicyclic component $U^\prime$ of $H^\prime$,
and $C_1^{\prime}$ is a subgraph of $U^\prime$.

Subcase 2.1: $uv, uw\not\in H^\prime$, then
$W(H^\prime)=4\cdot1\cdot1\cdot N$, for some constant value $N$.
$W(H)\geq(g(C_1)-1)\cdot1\cdot N$, so
$W(H)-W(H^\prime)\geq(g(C_1)-5)\cdot N\geq0$ by $g(C_1)\geq 5$.
Denote $\mathcal{H}^{\prime(21)}=\{H^\prime|u\in U^\prime, uv, uw\not\in H^\prime\}$,
then $\sum_{H^\prime\in\mathcal{H}^{\prime(21)}}[W(H)-W(H^\prime)]\geq 0$.

Subcase 2.2: $uv\not\in H^\prime, uw\in H^\prime$ or
$uw\not\in H^\prime, uv\in H^\prime$, then
$W(H^\prime)=4\cdot1\cdot N$, for some constant value $N$.
$W(H)\geq g(C_1)\cdot N$, so
$W(H)-W(H^\prime)\geq(g(C_1)-4)\cdot N >0$ by $g(C_1)\geq 5$.
Denote $\mathcal{H}^{\prime(22)}=\{H^\prime|u\in U^\prime,
uv\not\in H^\prime, uw\in H^\prime \mbox{ or }
uw\not\in H^\prime, uv\in H^\prime \}$,
then $\sum_{H^\prime\in\mathcal{H}^{\prime(22)}}[W(H)-W(H^\prime)]> 0$.

Case 3: $u$ is in an odd unicyclic component $U^\prime$ of $H^\prime$,
and $C_1^{\prime}$ is not a subgraph of $U^\prime$. Without loss of
generality, assume the subgraph $C_2$ of $G$ is a subgraph of $U^\prime$.

Subcase 3.1: $uv, uw\not\in H^\prime$, then
$W(H^\prime)=4\cdot1\cdot1\cdot N$, for some constant value $N$. Assume
the order of the tree in $H$ containing $v$, $w$ is $b_1, b_2 (b_1,b_2\geq1)$,
respectively. $W(H)=4\cdot b_1\cdot b_2\cdot N$, so
$W(H)-W(H^\prime)\geq0$.
Denote $\mathcal{H}^{\prime(31)}=\{H^\prime|u\in U^\prime, uv, uw\not\in H^\prime\}$,
then $\sum_{H^\prime\in\mathcal{H}^{\prime(31)}}[W(H)-W(H^\prime)]\geq 0$.

Subcase 3.2: $uv\not\in H^\prime, uw\in H^\prime$ or
$uw\not\in H^\prime, uv\in H^\prime$, then
$W(H^\prime)=4\cdot1\cdot N$, for some constant value $N$. Assume
the order of the tree $T$ in $H$ containing $w$ is $c$.
Since $v\in T$, thus $c\geq 1$.
$W(H)=4\cdot c\cdot N$, so
$W(H)-W(H^\prime)\geq 0$.
Denote $\mathcal{H}^{\prime(32)}=\{H^\prime|u\in U^\prime,
uv\not\in H^\prime, uw\in H^\prime \mbox{ or }
uw\not\in H^\prime, uv\in H^\prime \}$,
then $\sum_{H^\prime\in\mathcal{H}^{\prime(32)}}[W(H)-W(H^\prime)]\geq 0$.

Thus by summing over all possible subsets of $\mathcal{H}_i^\prime$,
$(\mathcal{H}_i^\prime=\mathcal{H}^{\prime(0)}\cup\mathcal{H}^{\prime(11)}\cup\mathcal{H}^{\prime(12)}
\cup\mathcal{H}^{\prime(13)}\cup\mathcal{H}^{\prime(21)}
\cup\mathcal{H}^{\prime(22)}\cup\mathcal{H}^{\prime(31)}\cup\mathcal{H}^{\prime(32)}),$
from Theorem~\ref{theorem1.2} and $f$ is an injection on the whole.
Then $$\varphi_i(G^\prime)=\sum_{H^\prime\in\mathcal{H}_i^\prime} W(H^\prime)
<\sum_{H\in\mathcal{H}_i^*} W(H)\leq\sum_{H\in\mathcal{H}_i} W(H)=\varphi_i(G)$$ holds
for $i=2,3,\cdots,n-1,n$.

For other cases in which $u, v, w$ satisfy, the discussion
is similar, thus we omit it.
\end{proof}

\begin{remark}
When $C_2$ exists, and $|\{u,v,w\}\cap V(C_2)|=0,\mbox{or } 1$, then after
performing transformation in Definition~\ref{definition3.3},
$g(C_1^\prime)=g(C_1)-2, g(C_2^\prime)=g(C_2)$.

When $C_3$ exists, and $|\{u,v,w\}\cap V(C_3)|=0$, then after
performing transformation in Definition~\ref{definition3.3},
$g(C_1^\prime)=g(C_1)-2, g(C_3^\prime)=g(C_3)$.

When $C_3$ $(\mbox{resp. } C_4)$ exists, and $|\{u,v,w\}\cap V(C_3)|=3$
$(\mbox{resp. } |\{u,v,w\}\cap V(C_4)|=3)$, then after
performing transformation in Definition~\ref{definition3.3},
$g(C_1^\prime)=g(C_1)-2, g(C_3^\prime)=g(C_3)-2$
$(\mbox{resp. } g(C_4^\prime)=g(C_4)-2)$.
\end{remark}

\section{The ordering of graphs in seven special sets}
For convenience, we define $B(3,3)$ as $B_1$, and $V(B_1)=\{x,u,v,w,z\}$,
where $d(u)=d(v)=d(w)=d(z)=2, d(x)=3$. The graph $B_1(a,b,c,d,e)$ is
obtained from $B_1$ by adding $a, b, c, d, e$ pendent vertices at
vertices $x, u, v, w, z$, respectively.

We define $B(3,4)$ as $B_2$, and $V(B_2)=\{u_1,u_2,u_3,u_4,u_5,u_6\}$,
where $d(u_2)=d(u_3)=d(u_4)=d(u_5)=d(u_6)=2, d(u_1)=4$. The graph $B_2(a,b,c,d,e,f)$ is
obtained from $B_2$ by adding $a, b, c, d, e, f$ pendent vertices at
vertices $u_1,u_2,u_3,u_4,u_5,u_6$, respectively.

We define $B(P_2,P_2,P_1)$ as $B_3$, and $V(B_3)=\{u,v,w,z\}$,
where $d(w)=d(z)=2, d(u)=d(v)=3$. The graph $B_3(a,b,c,d)$ is
obtained from $B_3$ by adding $a, b, c, d$ pendent vertices at
vertices $u, v, w, z$, respectively.

We define $B(P_3,P_2,P_1)$ as $B_4$, and $V(B_4)=\{u,v,w,z,x\}$,
where $d(w)=d(z)=d(x)=2, d(u)=d(v)=3$. The graph $B_4(a,b,c,d,e)$ is
obtained from $B_4$ by adding $a, b, c, d, e$ pendent vertices at
vertices $u, v, w, z, x$, respectively.

We define $B(P_3,P_2,P_2)$ as $B_5$, and $V(B_5)=\{u,v,u_1,v_1,w_1,w_2\}$,
where $d(u)=d(v)=3, d(u_1)=d(v_1)=d(w_1)=d(w_2)=2$. The graph $B_5(a,b,c,d,e,f)$ is
obtained from $B_5$ by adding $a, b, c, d, e, f$ pendent vertices at
vertices $u, v, u_1, v_1, w_1, w_2$, respectively. (See fig.3).

\begin{figure}
\begin{tikzpicture}
\node (v)[draw,shape=circle,inner sep=1.5pt,label=-45:$v$] at (-1,0){};
\node (v1)[draw,shape=circle,inner sep=1pt,label=-45:] at (-2,0){};
\node (v2)[draw,shape=circle,inner sep=1pt,label=-45:] at (-150:1.8){};

\node (x)[draw,shape=circle,inner sep=1.5pt,label=-90:$x$] at (0,0){};
\node (x1)[draw,shape=circle,inner sep=1pt,label=-45:] at (-60:1){};
\node (x2)[draw,shape=circle,inner sep=1pt,label=-45:] at (-120:1){};

\node (z)[draw,shape=circle,inner sep=1.5pt,label=45:$z$] at (1,0){};
\node (z1)[draw,shape=circle,inner sep=1pt,label=-45:] at (2,0){};
\node (z2)[draw,shape=circle,inner sep=1pt,label=-45:] at (-30:1.8){};

\node (w)[draw,shape=circle,inner sep=1.5pt,label=-45:$w$] at (60:1){};
\node (w1)[draw,shape=circle,inner sep=1pt,label=-45:] at (50:2){};
\node (w2)[draw,shape=circle,inner sep=1pt,label=-45:] at (75:2){};

\node (u)[draw,shape=circle,inner sep=1.5pt,label=45:$u$] at (120:1){};
\node (u1)[draw,shape=circle,inner sep=1pt,label=-45:] at (105:2){};
\node (u2)[draw,shape=circle,inner sep=1pt,label=-45:] at (135:2){};

\draw [loosely dotted,thick](v1)--(v2) node[pos=.5,sloped,below] {$c$};
\draw [loosely dotted,thick](w1)--(w2) node[pos=.5,sloped,above] {$d$};
\draw [loosely dotted,thick](x1)--(x2) node[pos=.5,sloped,below] {$a$};
\draw [loosely dotted,thick](z1)--(z2) node[pos=.5,sloped,below] {$e$};
\draw [loosely dotted,thick](u1)--(u2) node[pos=.5,sloped,above] {$b$};

\draw (x)--(x1)(x2)--(x) (x)--(z) (x)--(w)(x)--(u) (x)--(v)(v)--(u) (w)--(z)
(v)--(v1)(v2)--(v);
\draw (w)--(w1)(w)--(w2);
\draw (z)--(z1)(z2)--(z)(u)--(u1)(u2)--(u);

\node (t) at (0,-2) {$B_1$(a,b,c,d,e)};
\end{tikzpicture}
\begin{tikzpicture}
\node (v)[draw,shape=circle,inner sep=1.5pt,label=-45:$u_1$] at (0,0){};
\node (v1)[draw,shape=circle,inner sep=1pt,label=-45:] at (-60:1){};
\node (v2)[draw,shape=circle,inner sep=1pt,label=-45:] at (-120:1){};

\node (x)[draw,shape=circle,inner sep=1.5pt,label=-90:$u_6$] at (1,0){};
\node (x1)[draw,shape=circle,inner sep=1pt,label=-45:] at (2,0){};
\node (x2)[draw,shape=circle,inner sep=1pt,label=-45:] at (-30:1.7){};

\node (z)[draw,shape=circle,inner sep=1.5pt,label=-45:$u_5$] at (1,1){};
\node (z1)[draw,shape=circle,inner sep=1pt,label=-45:] at (30:2.3){};
\node (z2)[draw,shape=circle,inner sep=1pt,label=-45:] at (55:2.3){};

\node (w)[draw,shape=circle,inner sep=1.5pt,label=-45:$u_4$] at (0,1){};
\node (w1)[draw,shape=circle,inner sep=1pt,label=-45:] at (70:2){};
\node (w2)[draw,shape=circle,inner sep=1pt,label=-45:] at (95:2){};

\node (u)[draw,shape=circle,inner sep=1.5pt,label=180:$u_2$] at (120:1){};
\node (u1)[draw,shape=circle,inner sep=1pt,label=-45:] at (105:2){};
\node (u2)[draw,shape=circle,inner sep=1pt,label=-45:] at (135:2){};

\node (y)[draw,shape=circle,inner sep=1.5pt,label=-45:$u_3$] at (-1,0){};
\node (y1)[draw,shape=circle,inner sep=1pt,label=-45:] at (-2,0){};
\node (y2)[draw,shape=circle,inner sep=1pt,label=-45:] at (-150:1.7){};

\draw [loosely dotted,thick](v1)--(v2) node[pos=.5,sloped,below] {$a$};
\draw [loosely dotted,thick](w1)--(w2) node[pos=.5,sloped,above] {$d$};
\draw [loosely dotted,thick](x1)--(x2) node[pos=.5,sloped,below] {$f$};
\draw [loosely dotted,thick](z1)--(z2) node[pos=.5,sloped,above] {$e$};
\draw [loosely dotted,thick](u1)--(u2) node[pos=.5,sloped,above] {$b$};
\draw [loosely dotted,thick](y1)--(y2) node[pos=.5,sloped,below] {$c$};

\draw  (x)--(z) (v)--(w) (x)--(v)(v)--(u) (w)--(z)(y)--(v) (y)--(u);
\draw (v)--(v1)(v2)--(v)(x)--(x1)(x2)--(x);
\draw (w)--(w1)(w)--(w2)(y)--(y1)(y)--(y2);
\draw (z)--(z1)(z2)--(z)(u)--(u1)(u2)--(u);

\node (t) at (0,-2) {$B_2$(a,b,c,d,e,f)};
\end{tikzpicture}
\begin{tikzpicture}
\node (v)[draw,shape=circle,inner sep=1.5pt,label=-45:$v$] at (0,0){};
\node (v1)[draw,shape=circle,inner sep=1pt,label=-45:] at (-60:1){};
\node (v2)[draw,shape=circle,inner sep=1pt,label=-45:] at (-120:1){};

\node (z)[draw,shape=circle,inner sep=1.5pt,label=45:$z$] at (30:1){};
\node (z1)[draw,shape=circle,inner sep=1pt,label=-45:] at (15:2){};
\node (z2)[draw,shape=circle,inner sep=1pt,label=-45:] at (50:2){};

\node (w)[draw,shape=circle,inner sep=1.5pt,label=-135:$w$] at (150:1){};
\node (w1)[draw,shape=circle,inner sep=1pt,label=-45:] at (130:2){};
\node (w2)[draw,shape=circle,inner sep=1pt,label=-45:] at (165:2){};

\node (u)[draw,shape=circle,inner sep=1.5pt,label=45:$u$] at (0,1){};
\node (u1)[draw,shape=circle,inner sep=1pt,label=-45:] at (75:2){};
\node (u2)[draw,shape=circle,inner sep=1pt,label=-45:] at (110:2){};

\draw [loosely dotted,thick](v1)--(v2) node[pos=.5,sloped,below] {$b$};
\draw [loosely dotted,thick](w1)--(w2) node[pos=.5,sloped,above] {$c$};
\draw [loosely dotted,thick](z1)--(z2) node[pos=.5,sloped,above] {$d$};
\draw [loosely dotted,thick](u1)--(u2) node[pos=.5,sloped,above] {$a$};

\draw (v)--(z) (v)--(w)(w)--(u) (v)--(u) (u)--(z)
(v)--(v1)(v2)--(v);
\draw (w)--(w1)(w)--(w2);
\draw (z)--(z1)(z2)--(z)(u)--(u1)(u2)--(u);

\node (t) at (0,-2) {$B_3$(a,b,c,d)};
\end{tikzpicture}

\begin{tikzpicture}
\node (v)[draw,shape=circle,inner sep=1.5pt,label=-45:$v$] at (0,0){};
\node (v1)[draw,shape=circle,inner sep=1pt,label=-45:] at (-60:1){};
\node (v2)[draw,shape=circle,inner sep=1pt,label=-45:] at (-120:1){};

\node (x)[draw,shape=circle,inner sep=1.5pt,label=-90:$x$] at (150:1){};
\node (x1)[draw,shape=circle,inner sep=1pt,label=-45:] at (135:2){};
\node (x2)[draw,shape=circle,inner sep=1pt,label=-45:] at (170:2){};

\node (z)[draw,shape=circle,inner sep=1.5pt,label=45:$z$] at (1,0){};
\node (z1)[draw,shape=circle,inner sep=1pt,label=-45:] at (2,0){};
\node (z2)[draw,shape=circle,inner sep=1pt,label=-45:] at (-30:1.7){};

\node (w)[draw,shape=circle,inner sep=1.5pt,label=-45:$w$] at (1,1){};
\node (w1)[draw,shape=circle,inner sep=1pt,label=-45:] at (30:2.2){};
\node (w2)[draw,shape=circle,inner sep=1pt,label=-45:] at (60:2.2){};

\node (u)[draw,shape=circle,inner sep=1.5pt,label=45:$u$] at (0,1){};
\node (u1)[draw,shape=circle,inner sep=1pt,label=-45:] at (75:2){};
\node (u2)[draw,shape=circle,inner sep=1pt,label=-45:] at (105:2){};

\draw [loosely dotted,thick](v1)--(v2) node[pos=.5,sloped,below] {$b$};
\draw [loosely dotted,thick](w1)--(w2) node[pos=.5,sloped,above] {$c$};
\draw [loosely dotted,thick](x1)--(x2) node[pos=.5,sloped,below] {$e$};
\draw [loosely dotted,thick](z1)--(z2) node[pos=.5,sloped,below] {$d$};
\draw [loosely dotted,thick](u1)--(u2) node[pos=.5,sloped,above] {$a$};

\draw (x)--(u) (x)--(v)(v)--(z) (u)--(w)(v)--(u) (w)--(z);
\draw (w)--(w1)(w)--(w2)(x)--(x1)(x2)--(x) ;
\draw (z)--(z1)(z2)--(z)(u)--(u1)(u2)--(u) (v)--(v1)(v2)--(v);

\node (t) at (0,-2) {$B_4$(a,b,c,d,e)};
\end{tikzpicture}
\begin{tikzpicture}
\node (v)[draw,shape=circle,inner sep=1.5pt,label=-45:$v_1$] at (0,0){};
\node (v1)[draw,shape=circle,inner sep=1pt,label=-45:] at (-60:1.3){};
\node (v2)[draw,shape=circle,inner sep=1pt,label=-45:] at (-120:1.3){};

\node (x)[draw,shape=circle,inner sep=1.5pt,label=-90:$w_2$] at (2,0){};
\node (x1)[draw,shape=circle,inner sep=1pt,label=-45:] at (3,0){};
\node (x2)[draw,shape=circle,inner sep=1pt,label=-45:] at (-20:2.5){};

\node (z)[draw,shape=circle,inner sep=1.5pt,label=-45:$w_1$] at (2,2){};
\node (z1)[draw,shape=circle,inner sep=1pt,label=-45:] at (30:3.5){};
\node (z2)[draw,shape=circle,inner sep=1pt,label=-45:] at (45:4){};

\node (w)[draw,shape=circle,inner sep=1.5pt,label=-45:$u$] at (0,2){};
\node (w1)[draw,shape=circle,inner sep=1pt,label=-45:] at (70:3){};
\node (w2)[draw,shape=circle,inner sep=1pt,label=-45:] at (95:3){};

\node (u)[draw,shape=circle,inner sep=1.5pt,label=-135:$u_1$] at (150:2){};
\node (u1)[draw,shape=circle,inner sep=1pt,label=-45:] at (135:3){};
\node (u2)[draw,shape=circle,inner sep=1pt,label=-45:] at (160:3){};

\node (y)[draw,shape=circle,inner sep=1.5pt,label=-45:$v_1$] at (0,1){};
\node (y1)[draw,shape=circle,inner sep=1pt,label=-45:] at (30:1.2){};
\node (y2)[draw,shape=circle,inner sep=1pt,label=-45:] at (60:1.7){};

\draw [loosely dotted,thick](v1)--(v2) node[pos=.5,sloped,below] {$b$};

\draw [loosely dotted,thick](x1)--(x2) node[pos=.5,sloped,below] {$f$};
\draw [loosely dotted,thick](z1)--(z2) node[pos=.5,sloped,above] {$e$};
\draw [loosely dotted,thick](w1)--(w2) node[pos=.5,sloped,above] {$a$};
\draw [loosely dotted,thick](u1)--(u2) node[pos=.5,sloped,above] {$c$};
\draw [loosely dotted,thick](y1)--(y2) node[pos=.5,sloped,above] {$d$};

\draw  (v)--(x) (v)--(y) (u)--(v)(x)--(z) (w)--(z)(y)--(w) (w)--(u);
\draw (v)--(v1)(v2)--(v)(x)--(x1)(x2)--(x);
\draw (w)--(w1)(w)--(w2)(y)--(y1)(y)--(y2);
\draw (z)--(z1)(z2)--(z)(u)--(u1)(u2)--(u);

\node (t) at (0,-2) {$B_5$(a,b,c,d,e,f)};
\end{tikzpicture}
\begin{tikzpicture}
\node (v)[draw,shape=circle,inner sep=1.5pt,label=-45:$u_1$] at (0,1){};
\node (v1)[draw,shape=circle,inner sep=1pt,label=-45:] at (75:2){};
\node (v2)[draw,shape=circle,inner sep=1pt,label=-45:] at (105:2){};

\node (x)[draw,shape=circle,inner sep=1.5pt,label=-45:$w_3$] at (1,1){};
\node (x1)[draw,shape=circle,inner sep=1pt,label=-45:] at (30:2.3){};
\node (x2)[draw,shape=circle,inner sep=1pt,label=-45:] at (60:2.3){};

\node (z)[draw,shape=circle,inner sep=1.5pt,label=45:$w_4$] at (1,0){};
\node (z1)[draw,shape=circle,inner sep=1pt,label=-45:] at (2,0){};
\node (z2)[draw,shape=circle,inner sep=1pt,label=-45:] at (-30:2){};

\node (w)[draw,shape=circle,inner sep=1.5pt,label=-45:$u_2$] at (0,0){};
\node (w1)[draw,shape=circle,inner sep=1pt,label=-45:] at (-65:1){};
\node (w2)[draw,shape=circle,inner sep=1pt,label=-45:] at (-120:1){};

\node (u)[draw,shape=circle,inner sep=1.5pt,label=-45:$u_4$] at (-1,0){};
\node (u1)[draw,shape=circle,inner sep=1pt,label=-45:] at (-2,0){};
\node (u2)[draw,shape=circle,inner sep=1pt,label=-45:] at (-155:2){};

\node (y)[draw,shape=circle,inner sep=1.5pt,label=-45:$u_3$] at (-1,1){};
\node (y1)[draw,shape=circle,inner sep=1pt,label=-45:] at (-2,1){};
\node (y2)[draw,shape=circle,inner sep=1pt,label=-45:] at (130:2.3){};

\draw [loosely dotted,thick](v1)--(v2) node[pos=.5,sloped,above] {$a$};
\draw [loosely dotted,thick](w1)--(w2) node[pos=.5,sloped,below] {$b$};
\draw [loosely dotted,thick](x1)--(x2) node[pos=.5,sloped,above] {$e$};
\draw [loosely dotted,thick](z1)--(z2) node[pos=.5,sloped,below] {$f$};
\draw [loosely dotted,thick](u1)--(u2) node[pos=.5,sloped,below] {$d$};
\draw [loosely dotted,thick](y1)--(y2) node[pos=.5,sloped,above] {$c$};

\draw  (x)--(z) (v)--(w) (x)--(v) (w)--(z)(y)--(v) (y)--(u) (w)--(u);
\draw (v)--(v1)(v2)--(v)(x)--(x1)(x2)--(x);
\draw (w)--(w1)(w)--(w2)(y)--(y1)(y)--(y2);
\draw (z)--(z1)(z2)--(z)(u)--(u1)(u2)--(u);

\node (t) at (0,-2) {$B_6$(a,b,c,d,e,f)};
\end{tikzpicture}

\begin{tikzpicture}
\node (v)[draw,shape=circle,inner sep=1.5pt,label=-45:$v$] at (0,0){};
\node (v1)[draw,shape=circle,inner sep=1pt,label=-45:] at (-60:1){};
\node (v2)[draw,shape=circle,inner sep=1pt,label=-45:] at (-120:1){};

\node (x)[draw,shape=circle,inner sep=1.5pt,label=-90:$w_1$] at (30:2){};
\node (x1)[draw,shape=circle,inner sep=1pt,label=-45:] at (15:3){};
\node (x2)[draw,shape=circle,inner sep=1pt,label=-45:] at (30:3){};

\node (z)[draw,shape=circle,inner sep=1.5pt,label=-45:$u$] at (0,2){};
\node (z1)[draw,shape=circle,inner sep=1pt,label=-45:] at (75:3){};
\node (z2)[draw,shape=circle,inner sep=1pt,label=-45:] at (100:3){};

\node (w)[draw,shape=circle,inner sep=1.5pt,label=-45:$v_1$] at (0,1){};
\node (w1)[draw,shape=circle,inner sep=1pt,label=-45:] at (35:1){};
\node (w2)[draw,shape=circle,inner sep=1pt,label=-45:] at (60:1.5){};

\node (u)[draw,shape=circle,inner sep=1.5pt,label=180:$u_1$] at (135:1.6){};
\node (u1)[draw,shape=circle,inner sep=1pt,label=-45:] at (140:2.5){};
\node (u2)[draw,shape=circle,inner sep=1pt,label=-45:] at (160:2){};

\draw [loosely dotted,thick](v1)--(v2) node[pos=.5,sloped,below] {$b$};
\draw [loosely dotted,thick](x1)--(x2) node[pos=.5,sloped,above] {$e$};
\draw [loosely dotted,thick](z1)--(z2) node[pos=.5,sloped,above] {$a$};
\draw [loosely dotted,thick](w1)--(w2) node[pos=.5,sloped,above] {$d$};
\draw [loosely dotted,thick](u1)--(u2) node[pos=.5,sloped,below] {$c$};
\draw  (v)--(x) (v)--(w) (u)--(v)(x)--(z) (w)--(z)(u)--(z) ;
\draw (v)--(v1)(v2)--(v)(x)--(x1)(x2)--(x);
\draw (w)--(w1)(w)--(w2);
\draw (z)--(z1)(z2)--(z)(u)--(u1)(u2)--(u);

\node (t) at (0,-2) {$B_7$(a,b,c,d,e,f)};
\end{tikzpicture}
\begin{tikzpicture}
\node (v)[draw,shape=circle,inner sep=1.5pt,label=135:$u_1$] at (0,0){};
\node (v1)[draw,shape=circle,inner sep=1pt,label=-45:] at (-75:1){};
\node (v2)[draw,shape=circle,inner sep=1pt,label=-45:] at (-120:1){};

\node (x)[draw,shape=circle,inner sep=1.5pt,label=90:$v_4$] at (-45:1){};
\node (x1)[draw,shape=circle,inner sep=1pt,label=-45:] at (-60:2){};
\node (x2)[draw,shape=circle,inner sep=1pt,label=-45:] at (-30:2){};

\node (z)[draw,shape=circle,inner sep=1.5pt,label=180:$v_3$] at (1.6,0){};
\node (z1)[draw,shape=circle,inner sep=1pt,label=-45:] at (-10:2.5){};
\node (z2)[draw,shape=circle,inner sep=1pt,label=-45:] at (10:2.5){};

\node (w)[draw,shape=circle,inner sep=1.5pt,label=-90:$v_2$] at (45:1){};
\node (w1)[draw,shape=circle,inner sep=1pt,label=-45:] at (60:2){};
\node (w2)[draw,shape=circle,inner sep=1pt,label=-45:] at (30:2){};

\node (u)[draw,shape=circle,inner sep=1.5pt,label=135:$u_2$] at (0,1){};
\node (u1)[draw,shape=circle,inner sep=1pt,label=-45:] at (70:2){};
\node (u2)[draw,shape=circle,inner sep=1pt,label=-45:] at (95:2){};

\node (y)[draw,shape=circle,inner sep=1.5pt,label=-135:$u_3$] at (-1,1){};
\node (y1)[draw,shape=circle,inner sep=1pt,label=-45:] at (130:2.2){};
\node (y2)[draw,shape=circle,inner sep=1pt,label=-45:] at (155:2){};

\node (s)[draw,shape=circle,inner sep=1.5pt,label=135:$u_4$] at (-1,0){};
\node (s1)[draw,shape=circle,inner sep=1pt,label=-45:] at (-2,0){};
\node (s2)[draw,shape=circle,inner sep=1pt,label=-45:] at (-150:1.7){};

\draw [loosely dotted,thick](v1)--(v2) node[pos=.5,sloped,below] {$a$};
\draw [loosely dotted,thick](x1)--(x2) node[pos=.5,sloped,below] {$g$};
\draw [loosely dotted,thick](z1)--(z2) node[pos=.5,sloped,below] {$f$};
\draw [loosely dotted,thick](w1)--(w2) node[pos=.5,sloped,above] {$e$};
\draw [loosely dotted,thick](u1)--(u2) node[pos=.5,sloped,above] {$b$};
\draw [loosely dotted,thick](y1)--(y2) node[pos=.5,sloped,above] {$c$};
\draw [loosely dotted,thick](s1)--(s2) node[pos=.5,sloped,below] {$d$};

\draw  (v)--(x) (v)--(s) (u)--(v)(v)--(w)(x)--(z) (w)--(z)(u)--(y) (y)--(s);
\draw (v)--(v1)(v2)--(v)(x)--(x1)(x2)--(x);
\draw (w)--(w1)(w)--(w2)(y)--(y1)(y)--(y2);
\draw (z)--(z1)(z2)--(z)(u)--(u1)(u2)--(u)(s)--(s1)(s)--(s2);

\node (t) at (0,-2) {$B_8$(a,b,c,d,e,f,g)};
\end{tikzpicture}
\caption{Seven types of bicyclic graphs} \label{fig:pepper}
\end{figure}
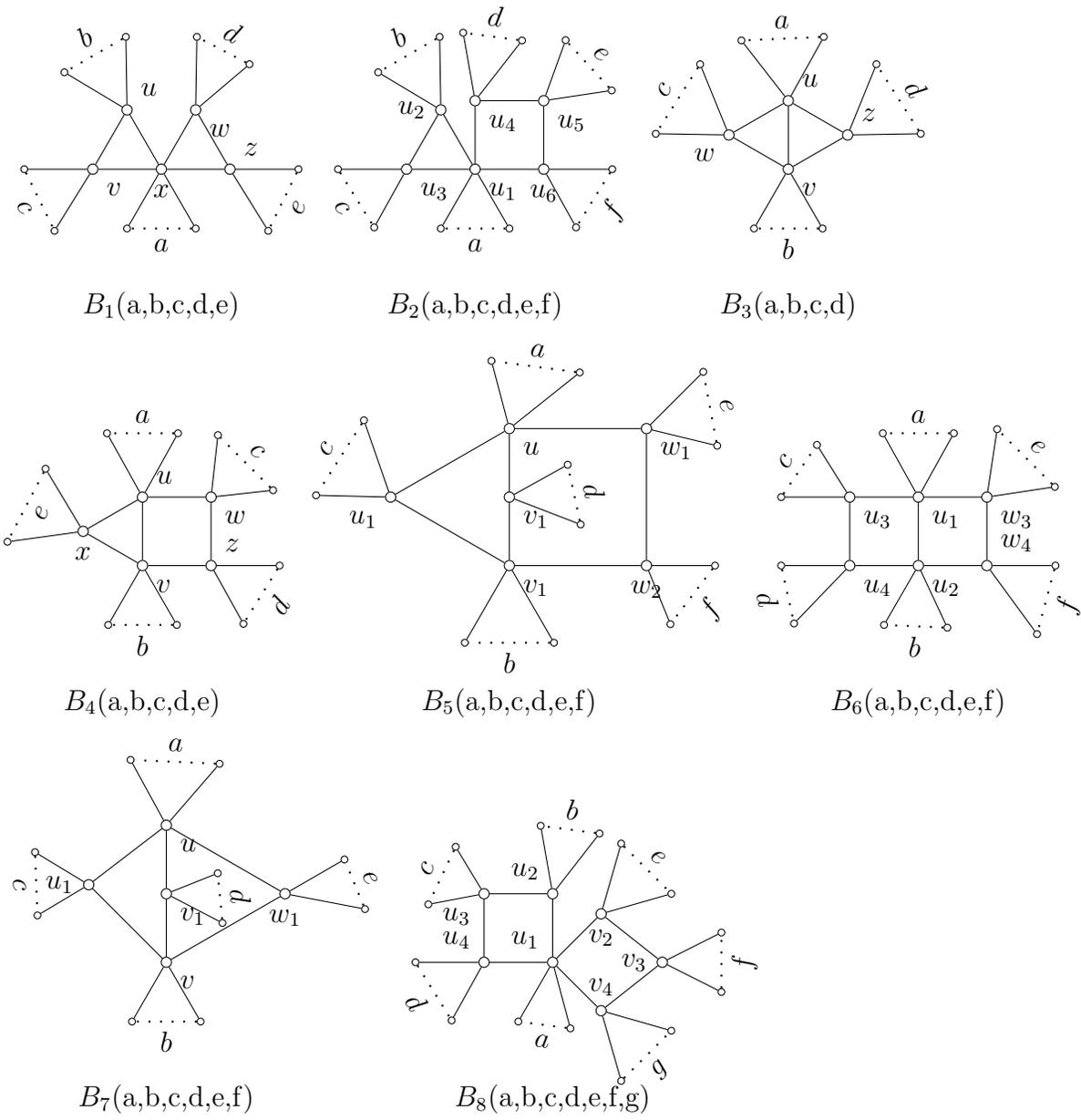
\begin{lemma}\label{lemma4.1}
Let $B_1(a,b,c,d,e)$ be the graph defined above, if we move all pendent
edges from vertices $u, v, w, z$ to vertex $x$. If $a+b+c+d\neq 0$,
then $$\varphi_i(B_1(a,b,c,d,e))\geq \varphi_i(B_1(a+b+c+d+e,0,0,0,0)), i=0,1,\cdots,n,$$
with equality holds if and only if $i\in\{0,1,n-1,n\}$.
\end{lemma}
\begin{proof}
The equality $\varphi_i(B_1(a,b,c,d,e))=\varphi_i(B_1(a+b+c+d+e,0,0,0,0)), i=0,1,n-1,n$
can be proved as the proof of Theorem~\ref{theorem3.2}.

For $2\leq i\leq n-2$, denote $\mathcal{H}_i^\prime$ and $\mathcal{H}_i$
the sets of all TU-subgraphs of $G^\prime$ and $G$ with exactly $i$ edges,
respectively. Let $\mathcal{H}_i^\prime=\mathcal{H}_i^{\prime 1}\cup \mathcal{H}_i^{\prime 2}
\cup \mathcal{H}_i^{\prime 3}\cup \mathcal{H}_i^{\prime 4}\cup \mathcal{H}_i^{\prime 5}$,
where $\mathcal{H}_i^{\prime j} (j=1,2,3,4,5)$ denotes vertices
$u, v, w, z, x$ belong to exactly $j$ components. $\mathcal{H}_i^j$
can be defined similarly.

For an arbitrary $H^\prime\in\mathcal{H}_i^\prime$,
assume $x$ is in a component $R^\prime$ of $H^\prime$.
Denote $a_1 (\mbox{ resp. } b_1,c_1,d_1,e_1)$ be the number of isolated
vertices in the set $N_{G^\prime}(x)\cap N_G(u)$
$( \mbox{ resp. } N_{G^\prime}(x)\cap N_G(v),
N_{G^\prime}(x)\cap N_G(w), N_{G^\prime}(x)\cap N_G(z), N_{G^\prime}(x)\cap (N_G(u)\setminus \{v,z\})).$
Write $A= a+1-a_1, B=b+1-b_1, C=c+1-c_1, D=d+1-d_1, E=e+1-e_1$,
without loss of generality, assume $N_G(u)\cap N_{R^\prime}(x)=\{u^1,\cdots, u^{b-b_1}\},
N_G(v)\cap N_{R^\prime}(x)=\{v^1,\cdots, v^{c-c_1}\}, N_G(w)\cap N_{R^\prime}(x)=\{w^1,\cdots, w^{d-d_1}\},
N_G(z)\cap N_{R^\prime}(x)=\{z^1,\cdots, z^{e-e_1}\}$, and
$(N_G(x)\setminus \{v,z\})\cap N_{R^\prime}(x)=\{x^1,\cdots, x^{a-a_1}\}$.
Define $H$ with $V(H)=V(H^\prime)$, $$E(H)=E(H^\prime)-\{xu^1,\cdots, xu^{b-b_1},
xv^1,\cdots, xv^{c-c_1}, xw^1,\cdots, xw^{d-d_1}, xz^1,\cdots, xz^{e-e_1}\}$$
$$+\{uu^1,\cdots, uu^{b-b_1},vv^1,\cdots, vv^{c-c_1}, ww^1,\cdots, ww^{d-d_1}, zz^1,\cdots, zz^{e-e_1}\}.$$
Then $H\in \mathcal{H}_i$, let $f: \mathcal{H}_i^\prime\rightarrow \mathcal{H}_i$, and
$\mathcal{H}_i^*=f(\mathcal{H}_i^\prime)=\{f(H^\prime)|H^\prime\in \mathcal{H}_i^\prime\}$.
It is easy to see that $H^\prime\in\mathcal{H}_i^{\prime j}\Leftrightarrow H\in \mathcal{H}_i^j,
j=1,2,3,4,5$.

If we include vertices $u, v, w, z, x$ in a component of $H^\prime$,
we have equal sizes of components in $H^\prime$ and $H$, respectively.
Then $W(H^\prime)=W(H)$.

Assume vertices $u, v, w, z, x$ belong to at least two components.
We distinguish the proof into two cases.

Case 1: When all components of $H^\prime$ are trees, and denote the set of this
kind of $H^\prime$ as $\mathcal{H}_{(1)}^\prime$. Then
$\sum_{H^\prime\in \mathcal{H}_{(1)}^\prime} (W(H)-W(H^\prime))\geq 0$,
with equality holds if and only if $i\in\{0,1,n-1,n\}$
by Lemma 3.4 in \cite{he2010}.

Case 2: When vertex $x$ belongs to an odd unicyclic component $U^\prime$,
without loss of generality, assume $C_1: uvx$ is a subgraph of $U^\prime$.

Subcase 2.1: Both $\{w,z\}$ belong to $U^\prime$, then $W(H^\prime)=4,
W(H)=4$. Denote the set of this
kind of $H^\prime$ as $\mathcal{H}_{(21)}^\prime$. Then
$\sum_{H^\prime\in \mathcal{H}_{(21)}^\prime} (W(H)-W(H^\prime))= 0$.

Subcase 2.2: At most one vertex in $\{w,z\}$ belongs to $U^\prime$.
Denote the set of this kind of $H^\prime$ as $\mathcal{H}_{(22)}^\prime$.
When $wx\in H^\prime, wz,xz \not\in H^\prime$, then $W(H^\prime)=4\cdot 1,
W(H)=4\cdot D$, thus $W(H)-W(H^\prime)\geq 0$.
When $xz\in H^\prime, wz,wz \not\in H^\prime$, then $W(H^\prime)=4\cdot 1,
W(H)=4\cdot C$, thus $W(H)-W(H^\prime)\geq 0$.
When $wz\in H^\prime, wx,xz \not\in H^\prime$, then $W(H^\prime)=4\cdot 2,
W(H)=4\cdot (C+D)$, thus $W(H)-W(H^\prime)\geq 0$.
When $wz,wx,xz \not\in H^\prime$, then $W(H^\prime)=4\cdot 1\cdot 1,
W(H)=4\cdot C\cdot D$, thus $W(H)-W(H^\prime)\geq 0$.
Then
$\sum_{H^\prime\in \mathcal{H}_{(22)}^\prime} (W(H)-W(H^\prime))\geq 0$.

Thus by summing over all possible subsets of $\mathcal{H}_i^\prime$,
$(\mathcal{H}_i^\prime=\mathcal{H}_{(1)}^{\prime}\cup\mathcal{H}_{(21)}^{\prime}\cup\mathcal{H}_{(22)}^{\prime}),$
from Theorem~\ref{theorem1.2} and $f$ is an injection on the whole.
Then $$\varphi_i(G^\prime)=\sum_{H^\prime\in\mathcal{H}_i^\prime} W(H^\prime)
<\sum_{H\in\mathcal{H}_i^*} W(H)\leq\sum_{H\in\mathcal{H}_i} W(H)=\varphi_i(G)$$ holds
for $i=2,3,\cdots,n-1$.

\end{proof}

Similar to the proof of Lemma~\ref{lemma4.1}, and by Lemma 3.2
and Lemma 3.3 in \cite{he2010}, the following Lemma holds:

\begin{lemma}\label{lemma4.2}
(1). Let $B_2(a,b,c,d,e,f)$ be the graph defined above, if we move all pendent
edges from vertices $u_2, u_3$ to vertex $u_1$. If $b+c\neq 0$,
then $$\varphi_i(B_2(a,b,c,d,e,f))\geq \varphi_i(B_2(a+b+c,0,0,d,e,f)), i=0,1,\cdots,n,$$
with equality holds if and only if $i\in\{0,1,n-1,n\}$.

(2). Let $B_2(a,0,0,d,e,f)$ be the graph defined above, if we move all pendent
edges from vertices $u_4, u_5, u_6$ to vertex $u_1$. If $d+e+f\neq 0$,
then $$\varphi_i(B_2(a,0,0,d,e,f))\geq \varphi_i(B_2(a+d+e+f,0,0,0,0,0)), i=0,1,\cdots,n,$$
with equality holds if and only if $i\in\{0,1,n-1,n\}$.

(3). Let $B_3(a,b,c,d)$ be the graph defined above, if we move all pendent
edges from vertices $v, w, z$ to vertex $u$. If $b+c+d\neq 0$,
then $$\varphi_i(B_3(a,b,c,d))\geq \varphi_i(B_3(a+b+c+d,0,0,0)), i=0,1,\cdots,n,$$
with equality holds if and only if $i\in\{0,1,n-1,n\}$.

(4). Let $B_4(a,b,c,d,e)$ be the graph defined above, if we move all pendent
edges from vertices $x$ to vertex $u$. If $e\neq 0$,
then $$\varphi_i(B_4(a,b,c,d,e))\geq \varphi_i(B_4(a+e,b,c,d,0)), i=0,1,\cdots,n,$$
with equality holds if and only if $i\in\{0,1,n-1,n\}$.

(5). Let $B_4(a,b,c,d,0)$ be the graph defined above, if we move all pendent
edges from vertices $w,z$ to vertex $u$. If $b+c\neq 0$,
then $$\varphi_i(B_4(a,b,c,d,0))\geq \varphi_i(B_4(a+c+d,b,0,0,0)), i=0,1,\cdots,n,$$
with equality holds if and only if $i\in\{0,1,n-1,n\}$.

(6). Let $B_4(a,b,0,0,0)$ be the graph defined above, if we move all pendent
edges from vertices $v$ to vertex $u$. If $b\neq 0$,
then $$\varphi_i(B_4(a,b,0,0,0))\geq \varphi_i(B_4(a+b,0,0,0,0)), i=0,1,\cdots,n,$$
with equality holds if and only if $i\in\{0,1,n-1,n\}$.

(7). Let $B_5(a,b,c,d,e,f)$ be the graph defined above, if we move all pendent
edges from vertices $\{u_1,u_2,v_1,v_2,w\}$ to vertex $u$. If $b+c+d+e+f\neq 0$,
then $$\varphi_i(B_5(a,b,c,d,e,f))\geq \varphi_i(B_5(a+b+c+d+e+f,0,0,0,0,0)), i=0,1,\cdots,n,$$
with equality holds if and only if $i\in\{0,1,n-1,n\}$.
\end{lemma}

For convenience, we define $B(P_3,P_3,P_1)$ as $B_6$, and $V(B_6)=\{u_1,u_2,u_3,u_4,w_3,w_4\}$,
where $d(u_3)=d(u_4)=d(w_3)=d(w_4)=2, d(u_1)=d(u_2)=3$. The graph $B_6(a,b,c,d,e,f)$ is
obtained from $B_6$ by adding $a, b, c, d, e, f$ pendent vertices at
vertices $u_1, u_2, u_3, u_4, w_3, w_4$, respectively.

We define $B(P_2,P_2,P_2)$ as $B_7$, and $V(B_7)=\{u,v,u_1,w_1,v_1\}$,
where $d(u_1)=d(v_1)=d(w_1)=2, d(u)=d(v)=3$. The graph $B_6(a,b,c,d,e)$ is
obtained from $B_7$ by adding $a, b, c, d, e$ pendent vertices at
vertices $u, v, u_1, w_1, v_1$, respectively.

We define $B(4,4)$ as $B_8$, and $V(B_8)=\{u_1,u_2,u_3,u_4,v_2,v_3,v_4\}$,
where $d(u_2)=d(u_3)=d(u_4)=d(v_2)=d(v_3)=d(v_4)=2, d(u_1)=4$.
The graph $B_8(a,b,c,d,e,f,g)$ is obtained from $B_8$ by adding
$a, b, c, d, e, f, g$ pendent vertices at $u_1, u_2, u_3, u_4, v_2, v_3, v_4$,
respectively. (See fig. 3).

\begin{lemma}\label{lemma4.3}
Let $B_6(a,b,c,d,e,f)$ be the graph defined above, if we move all pendent
edges from vertices $u_3, u_4$ to vertex $u_1$. If $c+d\neq 0$,
then $$\varphi_i(B_6(a,b,c,d,e,f))\geq \varphi_i(B_6(a+c+d,b,0,0,e,f)), i=0,1,\cdots,n,$$
with equality holds if and only if $i\in\{0,1,n-1,n\}$.
\end{lemma}

\begin{proof}
The equality $\varphi_i(B_6(a,b,c,d,e,f))=\varphi_i(B_6(a+c+d,b,0,0,e,f)), i=0,1,n-1,n$
can be proved as the proof of Theorem~\ref{theorem3.2}.

For $2\leq i\leq n-2$, denote $\mathcal{H}_i^\prime$ and $\mathcal{H}_i$
the sets of all TU-subgraphs of $G^\prime$ and $G$ with exactly $i$ edges,
respectively. Let $\mathcal{H}_i^\prime=\mathcal{H}_i^{\prime 1}\cup \mathcal{H}_i^{\prime 2}
\cup \mathcal{H}_i^{\prime 3}\cup \mathcal{H}_i^{\prime 4}$,
where $\mathcal{H}_i^{\prime j} (j=1,2,3,4)$ denotes vertices
$u_1, u_2, u_3, u_4$ belong to exactly $j$ components. $\mathcal{H}_i^j$
can be defined similarly.

For an arbitrary $H^\prime\in\mathcal{H}_i^\prime$,
assume $u_1$ is in a component $R^\prime$ of $H^\prime$.
Assume $(N_G(u_3)\setminus\{u_2,u_4\})\cap N_{R^\prime}(u_1)=\{u_3^{i_1},\cdots, u_3^{i_r}\}$,
$N_G(u_4)\cap N_{R^\prime}(u_1)=\{u_4^{i_1},\cdots, u_4^{i_s}\}$,
where $0\leq r\leq \mbox{min} \{d_G(u_3)-2, |V(R^\prime)|-1\},
0\leq s\leq \mbox{min} \{d_G(u_4)-2, |V(R^\prime)|-1\}$.
Define $H$ with $V(H)=V(H^\prime)$, $E(H)=E(H^\prime)-\{u_1u_3^{i_1}-\cdots -u_1u_3^{i_r}-
u_1u_4^{i_1}-\cdots -u_1u_4^{i_s}+u_3u_3^{i_1}+\cdots +u_3u_3^{i_r}+
u_4u_4^{i_1}+\cdots +u_4u_4^{i_s}\}$.
Then $H\in \mathcal{H}_i$, let $f: \mathcal{H}_i^\prime\rightarrow \mathcal{H}_i$, and
$\mathcal{H}_i^*=f(\mathcal{H}_i^\prime)=\{f(H^\prime)|H^\prime\in \mathcal{H}_i^\prime\}$.
It is easy to see that $H^\prime\in\mathcal{H}_i^{\prime j}\Leftrightarrow H\in \mathcal{H}_i^j,
j=1,2,3,4$.

Let $A+1$ be the order of the subgraph of $H^\prime-u_1u_2-u_1u_4$
which contains $u_1$ and excluding
the vertices in $N_G(u_3)\cup N_G(u_4)$
and $B+1$ be the order of the subgraph of $H^\prime-u_1u_2-u_2u_3$
which contains $u_2$,
and denote $C=|N_{R^\prime}(u_1)\cap N_G(u_4)|,
D=|N_{R^\prime}(u_1)\cap (N_G(u_3)\setminus\{u_2,u_4\})|$,
$(A,B,C,D\geq 0, \mbox{ and } A,B,C,D \mbox{ is fixed})$.
Denote $N$ the product of
all the orders of components of $H^\prime$ except the components containing
$u_1, u_2, u_3, u_4$.

If we include vertices $u_1, u_3, u_4$ in a component of $H^\prime$,
we have equal sizes of components in $H^\prime$ and $H$, respectively.
Then $W(H^\prime)=W(H)$.

Assume vertices $u_1, u_3, u_4$ belong to at least two components.
we distinguish the proof into three cases.

Case 1: $H^\prime\in \mathcal{H}_i^{\prime 2}$.

Subcase 1.1: $u_1u_2, u_3u_4\in H^\prime$, $u_2u_3, u_1u_4\not\in H^\prime$,
then $$W(H)-W(H^\prime)=[(A+B+2)(C+D+2)-2(A+B+C+D+2)]\cdot N.$$

Subcase 1.2: $u_1u_2, u_3u_4\not\in H^\prime$, $u_2u_3, u_1u_4\in H^\prime$,
then $$W(H)-W(H^\prime)=[(A+D+2)(B+C+2)-(B+2)(A+C+D+2)]\cdot N.$$

Subcase 1.3: $u_1u_2, u_2u_3\in H^\prime$, $u_3u_4, u_1u_4\not\in H^\prime$,
then $$W(H)-W(H^\prime)=[(A+B+C+3)(D+1)-(A+B+C+D+3)]\cdot N.$$

Subcase 1.4: $u_1u_2, u_1u_4\in H^\prime$, $u_3u_4, u_2u_3\not\in H^\prime$,
then $$W(H)-W(H^\prime)=[(A+B+D+3)(C+1)-(A+B+C+D+3)]\cdot N.$$

Subcase 1.5: $u_2u_3, u_3u_4\in H^\prime$, $u_1u_2, u_1u_4\not\in H^\prime$,
then $$W(H)-W(H^\prime)=[(B+C+D+3)(A+1)-(A+C+D+1)(B+3)]\cdot N.$$

We can find a bijection between every two subsets of the above subcases.
Hence $$\sum_{H\in \mathcal{H}_i^2} W(H)- \sum_{H^\prime\in \mathcal{H}_i^{\prime 2}} W(H^\prime)\geq 0.$$

Case 2: $H^\prime\in \mathcal{H}_i^{\prime 3}$.

Subcase 2.1: $u_1u_2\in H^\prime$, $u_2u_3, u_3u_4, u_1u_4\not\in H^\prime$,
then $$W(H)-W(H^\prime)=[(A+B+2)(C+1)(D+1)-(A+B+C+D+2)]\cdot N.$$

Subcase 2.2: $u_2u_3\in H^\prime$, $u_1u_2, u_3u_4, u_1u_4\not\in H^\prime$,
then $$W(H)-W(H^\prime)=[(A+1)(D+1)(B+C+2)-(B+2)(A+C+D+1)]\cdot N.$$

Subcase 2.3: $u_3u_4 \in H^\prime$, $ u_1u_2, u_2u_3, u_1u_4\not\in H^\prime$,
then $$W(H)-W(H^\prime)=[(A+1)(B+1)(C+D+2)-2(B+1)(A+C+D+1)]\cdot N.$$

Subcase 2.4: $u_1u_4\in H^\prime$, $u_1u_2, u_2u_3, u_3u_4\not\in H^\prime$,
then $$W(H)-W(H^\prime)=[(A+D+2)(B+1)(C+1)-(A+C+D+2)(B+1)]\cdot N.$$

We can find a bijection between every two subsets of the above subcases.
Hence $$\sum_{H\in \mathcal{H}_i^3} W(H)- \sum_{H^\prime\in \mathcal{H}_i^{\prime 3}} W(H^\prime)\geq 0.$$

Case 3: $H^\prime\in \mathcal{H}_i^{\prime 4}$,
$u_1u_2, u_2u_3, u_3u_4, u_1u_4\not\in H^\prime$,
then $W(H)-W(H^\prime)=[(A+1)(B+1)(C+1)(D+1)-(A+C+D+1)(B+1)]\cdot N
=(B+1)(ACD+AC+CD+AD)\cdot N.$ Thus
$$\sum_{H\in \mathcal{H}_i^4} W(H)- \sum_{H^\prime\in \mathcal{H}_i^{\prime 4}} W(H^\prime)\geq 0.$$

Combining the above cases and
$\sum_{H\in \mathcal{H}_i^1} W(H)= \sum_{H^\prime\in \mathcal{H}_i^{\prime 1}} W(H^\prime)$,
from Theorem~\ref{theorem1.2} and $f$ is an injection on the whole.
Then $$\varphi_i(G^\prime)=\sum_{H^\prime\in\mathcal{H}_i^\prime} W(H^\prime)
<\sum_{H\in\mathcal{H}_i^*} W(H)\leq\sum_{H\in\mathcal{H}_i} W(H)=\varphi_i(G)$$ holds
for $i=2,3,\cdots,n-1$.

\end{proof}

Similar to the proof of Lemma~\ref{lemma4.3}, we have the following Lemma.

\begin{lemma}\label{lemma4.4}
(1). Let $B_6(a,b,0,0,0,0)$ be the graph defined above, if we move all pendent
edges from vertices $u_2$ to vertex $u_1$. If $b\neq 0$,
then $$\varphi_i(B_6(a,b,0,0,0,0))\geq \varphi_i(B_6(a+b,0,0,0,0,0)), i=0,1,\cdots,n,$$
with equality holds if and only if $i\in\{0,1,n-1,n\}$.

(2). Let $B_7(a,b,c,d,e)$ be the graph defined above, if we move all pendent
edges from vertices $u_1$ to vertex $u$. If $c\neq 0$,
then $$\varphi_i(B_7(a,b,c,d,e))\geq \varphi_i(B_7(a+c,b,0,d,e)), i=0,1,\cdots,n,$$
with equality holds if and only if $i\in\{0,1,n-1,n\}$.

(3). Let $B_7(a,b,0,0,0)$ be the graph defined above, if we move all pendent
edges from vertices $v$ to vertex $u$. If $b\neq 0$,
then $$\varphi_i(B_7(a,b,0,0,0))\geq \varphi_i(B_7(a+b,0,0,0,0)), i=0,1,\cdots,n,$$
with equality holds if and only if $i\in\{0,1,n-1,n\}$.

(4). Let $B_8(a,b,c,d,e,f,g)$ be the graph defined above, if we move all pendent
edges from vertices $u_2,u_3,u_4$ to vertex $u_1$. If $b+c+d\neq 0$,
then $$\varphi_i(B_8(a,b,c,d,e,f,g))\geq \varphi_i(B_8(a+b+c+d,0,0,0,e,f,g)), i=0,1,\cdots,n,$$
with equality holds if and only if $i\in\{0,1,n-1,n\}$.

(5). Let $B_8(a,0,0,0,e,f,g)$ be the graph defined above, if we move all pendent
edges from vertices $v_2,v_3,v_4$ to vertex $u_1$. If $e+f+g\neq 0$,
then $$\varphi_i(B_8(a,0,0,0,e,f,g))\geq \varphi_i(B_8(a+e+f+g,0,0,0,0,0,0)), i=0,1,\cdots,n,$$
with equality holds if and only if $i\in\{0,1,n-1,n\}$.
\end{lemma}

For any graph $G$ and $v\in V(G)$, let $Q_{G|v}(x)$
be the principal submatrix of $Q_G(x)$ obtained by deleting the row
and column corresponding to the vertex $v$. Similar to the proof of
Theorem~\ref{theorem2.2} and Theorem~\ref{theorem2.3}, we can prove the
following two lemmas.

\begin{lemma}\label{lemma4.5}
If $G=G_1|u : G_2|v$, then
$Q_G(x)=Q_{G_1}(x)Q_{G_2}(x)-Q_{G_1}(x)Q_{G_2|v}(x)-Q_{G_2}(x)Q_{G_1|u}(x)$.
\end{lemma}

\begin{lemma}\label{lemma4.6}
If $G$ be a connected graph with $n$ vertices which consists of a subgraph $H (|V(H)|\geq 2)$
and $n-|V(H)|$ pendent vertices attached to a vertex $v$ in $H$, then
$Q_G(x)=(x-1)^{(n-|V(H)|)}Q_{H}(x)-(n-|V(H)|)x(x-1)^{(n-|V(H)|-1)}Q_{H|v}(x)$.
\end{lemma}

Let $f(x)=\sum_{i=0}^n (-1)^ia_ix^{n-i}, g(x)=\sum_{j=0}^m (-1)^ja_jx^{m-j}, a_i>0, b_j>0$.
Then it is easy to see $f(x)g(x)=f(x)=\sum_{k=0}^{m+n} (-1)^k \sum_{i=0}^k a_ib_{k-i}x^{n+m-k}$
has coefficients alternate with positive and negative.

By using Lemma~\ref{lemma4.5} and Lemma~\ref{lemma4.6}, we can
compute the signless Laplacian polynomials of seven special $n$-vertex bicyclic graphs.
For convenience, write $Q_G(x)$ as $Q(G, x)$.

$Q(B_1(n-5,0,0,0,0),x)=(x-1)^{n-4}(x-3)[x^3-(n+3)x^2+3nx-8],$

$Q(B_2(n-6,0,0,0,0,0),x)=(x-1)^{n-6}(x-2)[x^5-(n+6)x^4+7(n+1)x^3-2(7n-1)x^2+2(3n+8)x-8],$

$Q(B_3(n-4,0,0,0),x)=(x-1)^{n-4}(x-2)[x^3-(n+4)x^2+4nx-8],$

$Q(B_4(n-5,0,0,0,0),x)=(x-1)^{n-6}[x^6-(n+8)x^5+9(n+2)x^4-(27n+10)x^3+(31n+10)x^2-(11n+32)x+16],$

$Q(B_5(n-6,0,0,0,0,0),x)=(x-1)^{n-7}(x-2)[x^6-(n+7)x^5+(9n+8)x^4-(26n-22)x^3+(27n-30)x^2-(8n+8)x+8],$

$Q(B_6(n-6,0,0,0,0,0),x)=x(x-1)^{n-6}(x-3)[x^4-(n+5)x^3+(7n-1)x^2-(13n-17)x+5n],$

$Q(B_7(n-5,0,0,0,0),x)=x(x-1)^{n-6}(x-2)^2[x^3-(n+4)x^2+(5n-2)x-3n],$

$Q(B_8(n-7,0,0,0,0,0,0),x)=x(x-1)^{n-8}(x-2)^2(x^2-4x+2)[x^3-(n+2)x^2+2(2n-3)x-2n].$

Then we have
\begin{eqnarray}
&Q(B_1(n-5,0,0,0,0),x)-Q(B_3(n-4,0,0,0),x)\nonumber\\
&=(x-1)^{n-4}(x^2-n x+8).
\end{eqnarray}
\begin{eqnarray}
&Q(B_2(n-6,0,0,0,0,0),x)-Q(B_4(n-5,0,0,0,0),x)\nonumber\\
&=x(x-1)^{n-6}[x^3-(n+2)x^2+(3n+2)x-(n+8)].
\end{eqnarray}
\begin{eqnarray}
&Q(B_4(n-5,0,0,0,0),x)-Q(B_3(n-4,0,0,0),x)\nonumber\\
&=x(x-1)^{n-6}[(n-3)x^3-(6n-20)x^2+(9n-30)x-(3n-8)].
\end{eqnarray}
\begin{eqnarray}
&Q(B_5(n-6,0,0,0,0,0),x)-Q(B_3(n-4,0,0,0),x)\nonumber\\
&=x(x-2)(x-1)^{n-7}[(2n-7)x^3-(11n-43)x^2+(14n-58)x-(4n-16)].
\end{eqnarray}
Thus $B_3(n-4,0,0,0)$ has minimum signless Laplacian coefficients in the set
$\{B_1(n-5,0,0,0,0), B_2(n-6,0,0,0,0,0), B_3(n-4,0,0,0), B_4(n-5,0,0,0,0)\}$.

Moreover,
\begin{eqnarray}
&Q(B_8(n-7,0,0,0,0,0,0),x)-Q(B_6(n-6,0,0,0,0,0),x)\nonumber\\
&=x(x-1)^{n-8}[x^5-(n+4)x^4+(6n+1)x^3-(11n-6)x^2+3(2n+1)x-n].
\end{eqnarray}
\begin{eqnarray}
&Q(B_6(n-6,0,0,0,0,0),x)-Q(B_7(n-5,0,0,0,0),x)\nonumber\\
&=x(x-1)^{n-6}[(n-4)x^3-(7n-28)x^2+(12n-43)x-3n].
\end{eqnarray}
Thus $B_7(n-5,0,0,0,0)$ has minimum signless Laplacian coefficients in the set
$\{B_6(n-6,0,0,0,0,0), B_7(n-5,0,0,0,0), B_8(n-7,0,0,0,0,0,0)\}$.

For convenience, write $B_3(n-4,0,0,0)$ as $B_n^1$, $B_7(n-5,0,0,0,0)$ as $B_n^2$.

\section{The graphs which have minimum signless Laplacian coefficients in $\mathcal{B}^1(n)$ and $\mathcal{B}^2(n)$}
\begin{theorem}\label{theorem5.1}
In the set $\mathcal{B}^1(n)$, for $G\in\mathcal{B}^1(n)$, $G\not\cong B_n^1$,
$\varphi_i(G)\geq \varphi_i(B_n^1), i=0,1,\cdots,n$.
With equality if and only if either $i\in\{0,1,n-1,n\}$ when
$\overline{G}\cong B(P_2,P_2,P_1)$ or $i\in\{0,1\}$ otherwise.
\end{theorem}

\begin{proof}
Let $G$ be an arbitrary graph in $\mathcal{B}^1(n)$.
We need to prove after series of transformations, $G$ will
become to $B_n^1$ and $B_n^1$ has minimum signless Laplacian coefficients
in $\mathcal{B}^1(n)$.

Step 1: When there is a non-pendent edge $uv$ which is not on the cycle.
By performing the transformation of Definition~\ref{definition3.1} to $uv$,
we have $G_{uv}\in\mathcal{B}^1(n)$, and
$\varphi_i(G)>\varphi_i(G_{uv})$, $i=2,3,\cdots,n-1$ by Theorem~\ref{theorem3.2}.

After performing Step 1 consecutively, it is easy to see
that all cut edges are pendent edges in the resulting graph.

Step 2: Since $G$ is a bicyclic graph, there are two minimal
cycles $C_1, C_2$ in $G$.
For $u, v, w\in V(C_1)$ and satisfy the claim of Definition~\ref{definition3.3}, if
$|V(C_1)|\geq 5$, and $|V(C_2)\cap V(C_1)|\geq 3$, by the
assumption $|V(C_2)\cap V(C_1)|=\mbox{min }\{k,l,m\}$,
we have $|V(C_2)|\geq 5$. Then we perform the transformation
of Definition~\ref{definition3.3} to $u, v, w\in V(C_1)\cap V(C_2)$, and obtain
$G^\prime$, by Theorem~\ref{theorem3.4},
$\varphi_i(G)> \varphi_i(G^\prime), i=2, 3,\cdots, n$,
and the lengths of the cycles of $G$
are decreased by $2$.

If $|V(C_1)|\geq 5$, and $|V(C_2)\cap V(C_1)|\leq 3$.
Assume $u, v, w$ satisfy the claim of Definition~\ref{definition3.3}
and at most one of $\{u,v,w\}$ belongs to $V(C_1)\cap V(C_2)$,
then we perform the transformation
of Definition~\ref{definition3.3} to $u, v, w$, and obtain
$G^\prime$, by Theorem~\ref{theorem3.4},
$\varphi_i(G)> \varphi_i(G^\prime), i=2, 3,\cdots, n$,
and the length of $C_1$ is decreased by $2$, while
the length of $C_2$ keeps unchanged.

Therefore, after taking Step 2 consecutively, we obtain five types of
graphs which has been discussed from Lemma~\ref{lemma4.1} and Lemma~\ref{lemma4.2},
then by comparing the signless Laplacian polynomials of the resulting five
extremal graphs, it is easy to see that $B_n^1$ has
minimum signless Laplacian coefficients in $\mathcal{B}^1(n)$ by equations (1)-(4).
\end{proof}

By Theorem~\ref{theorem1.3} and Theorem~\ref{theorem5.1},
we obtained the following corollary.

\begin{corollary}\label{corollary5.2}
In the set of all n-vertex bicyclic graphs in $\mathcal{B}^1(n)$,
$B_n^1$ is the unique graph with the minimal $IE$.
\end{corollary}

\begin{theorem}\label{theorem5.3}
In the set $\mathcal{B}^2(n)$, for $G\in\mathcal{B}^2(n)$, $G\not\cong B_n^2$,
$\varphi_i(G)\geq \varphi_i(B_n^2), i=0,1,\cdots,n$.
With equality if and only if either $i\in\{0,1,n-1,n\}$ when
$\overline{G}\cong B(P_2,P_2,P_2)$ or $i\in\{0,1\}$ otherwise.
\end{theorem}

The proof is similar to the proof of Theorem~\ref{theorem5.1}.

By Theorem~\ref{theorem1.3} and Theorem~\ref{theorem5.3},
we obtained the following corollary.

\begin{corollary}\label{corollary5.4}
In the set of all n-vertex bicyclic graphs in $\mathcal{B}^2(n)$,
$B_n^2$ is the unique graph with the minimal $IE$.
\end{corollary}

From Corollary~\ref{corollary5.2} and Corollary~\ref{corollary5.4}
we immediately get the following result.

\begin{corollary}\label{corollary5.5}
If $n\leq 30$, then for $G\in \mathcal{B}(n)$ we have $IE(G)\geq IE(B_n^2)$,
with equality if and only if $G\cong B_n^2$.
If $n\geq 31$, then for $G\in \mathcal{B}(n)$ we have $IE(G)\geq IE(B_n^1)$,
with equality if and only if $G\cong B_n^1$.
\end{corollary}

\begin{proof}
From Corollary~\ref{corollary5.2} and Corollary~\ref{corollary5.4},
for $G\in \mathcal{B}(n)$ we have $IE(G)\geq \mbox{min }\{IE(B_n^1), IE(B_n^2)\}$,
with equality if and only if $G\cong B_n^1$ or $G\cong B_n^2$.

From $Q(B_n^1,x)$ and $Q(B_n^2,x)$ in Section 4, we have
$$IE(B_n^1)=(n-4)+\sqrt{2}+\sqrt{\alpha_1}+\sqrt{\alpha_2}+\sqrt{\alpha_3},$$
$$IE(B_n^2)=(n-6)+2\sqrt{2}+\sqrt{\beta_1}+\sqrt{\beta_2}+\sqrt{\beta_3},$$
where $\alpha_1\geq \alpha_2\geq \alpha_3$ are the roots of $x^3-(n+4)x^2+4nx-8=0$,
$\beta_1\geq \beta_2\geq \beta_3$ are the roots of $x^3-(n+4)x^2+(5n-2)x-3n=0$.

For $n\leq 30$, by Matlab 7.0 it is easy to see $IE(B_n^1)>IE(B_n^2)$ holds.

For $n\geq 31$, it is easy to see that
$n\leq \alpha_1\leq n+0.01$, $3.93\leq \alpha_2\leq 4$, $0\leq \alpha_3\leq 0.066$,
and $n-1 \leq \beta_1 \leq n-0.995$, $4.27\leq \beta_2 \leq 4.31$,
$0.697 \leq \beta_3 \leq 0.726$,
and $0.5899\leq \sum_{i=1}^3 (\sqrt{\beta_i}-\sqrt{\alpha_i})\leq 1$.

Then we have $IE(B_n^2)-IE(B_n^1)=\sqrt{2}-2+\sum_{i=1}^3 (\sqrt{\beta_i}-\sqrt{\alpha_i})
\geq 0$.
\end{proof}

\end{document}